\documentclass[12pt,english,reqno]{amsart}
\usepackage[latin9]{inputenc}
\usepackage{verbatim}
\usepackage{textcomp}
\usepackage{amsthm}
\usepackage{amstext}
\usepackage{amssymb}
\usepackage{graphicx}
\usepackage{esint}
\usepackage[normalem]{ulem}
\date{\today}

\makeatletter
\newcounter{cprop}[section]

\usepackage{amscd}\usepackage{amsthm}\usepackage[english]{babel}
\usepackage[arrow,matrix]{xy}\usepackage{cite}\usepackage{bbm}
\@ifundefined{definecolor}
 {\usepackage{color}}{}
\usepackage{MnSymbol}

\theoremstyle{definition}
\newtheorem{thm}[cprop]{Theorem}
\newtheorem{rem}[cprop]{Remark}
\newtheorem{lemma}[cprop]{Lemma}
\newtheorem{pro}[cprop]{Proposition}
\newtheorem{cor}[cprop]{Corollary}

\newtheorem{rk}[cprop]{Remark}
\newtheorem{defn}[cprop]{Definition}
\newtheorem{ex}[cprop]{Example}

\numberwithin{equation}{section}

\newcommand{\F}{{\mathcal F}}
\newcommand{\M}{{\mathcal M}}

\newcommand{\B}{{\mathcal B}}

\newcommand{\dd}{{\mathrm{d}}}

\newcommand{\EE}{{\mathcal E}}
\newcommand{\E}{{\mathbb E}}

\newcommand{\w}{{\bf w}}
\newcommand{\s}{{\sigma}}

\renewcommand{\S}{{\mathcal S}}

\newcommand{\bea}{\begin{eqnarray}}
\newcommand{\eea}{\end{eqnarray}}

\newcommand{\N}{\mathbb N}

\newcommand{\Q}{\mathbb{Q}}
\newcommand{\R}{\mathbb{R}}

\newcommand{\PP}{\mathbb{P}}
\newcommand{\FF}{\mathbb{F}}

\def\d {\triangle}
\def\t {\theta}
\def\w {\omega}
\def\> {\Rightarrow}
\def\0 {\emptyset}

\def\a {\alpha}
\def\l {\lambda}
\def\d {\delta}
\def\e {\varepsilon}

\def\b {\beta}
\def\g {\gamma}
\def\s {\sigma}

\newcommand{\diam}{\mathrm{diam}}
\newcommand{\range}{\mathrm{range}}

\newcommand{\supp}{\mathrm{supp}}

\definecolor{gray}{rgb}{0.75, 0.75, 0.75}
\newcommand{\gray}[1]{}
\newcommand{\blue}[1]{#1}
\makeatother

\usepackage{babel}

\begin{document}
\title[Synchronization by noise]{Synchronization by noise}

\begin{abstract}
We provide sufficient conditions for synchronization by noise, i.e.\ under these conditions we prove that weak random attractors for random dynamical systems consist of single random points. In the case of SDE with additive noise, these conditions are also essentially necessary. In addition, we provide sufficient conditions for the existence of a minimal weak point random attractor consisting of a single random point. As a result, synchronization by noise is proven for a large class of SDE with additive noise. In particular, we prove that the random attractor for an SDE with drift 
given by a (multidimensional) double-well potential and additive noise consists of a single random point. All examples treated in \cite{T08} are also included.
\end{abstract}

\author[F. Flandoli]{Franco Flandoli}
\address{Dipartimento di Matematica\\
Largo Bruno Pontecorvo 5\\
56127 Pisa\\
Italy}
\email{flandoli@dma.unipi.it}

\author[B. Gess]{Benjamin Gess}
\address{Max-Planck Institute for Mathematics in the Sciences \\
04103 Leipzig\\
Germany }
\email{bgess@mis.mpg.de}

\author[M. Scheutzow]{Michael Scheutzow}
\address{Institut f\"ur Mathematik, MA 7-5\\
Technische Universit\"at Berlin\\
10623 Berlin \\
Germany}
\email{ms@math.tu-berlin.de}



\thanks{B.G. has been partially supported by the research project ``Random dynamical systems and regularization by noise for stochastic partial differential equations'' funded by the German Research Foundation.}

\maketitle

\section{Introduction}
In this paper we introduce new, checkable conditions for synchronization by noise for general white noise, random dynamical systems (RDS) $\varphi$ on complete, separable metric spaces $E$. Here, synchronization by noise means that there is a (weak) random attractor\footnote{For notation and background on RDS see Section \ref{sec:prelim_not} below.} $A$ for $\varphi$ consisting of a single random point, i.e. $A(\omega)=\{a(\omega)\}$ a.s.\ and thus the long-time dynamics are asymptotically globally stable. In particular, for each $x,y\in E$ it follows that
  $$\lim_{t\to\infty} d(\varphi_t(\omega,x),\varphi_t(\omega,y)) = 0$$
in probability. 

We are especially interested in SDE with additive noise
\begin{equation}\label{eq:SDE_intro}
   dX_t = b(X_t)dt+\s dW_t\quad\text{on }\R^d
\end{equation}
with $\s>0$, for choices of $b$ such that the deterministic dynamics corresponding to $\s=0$ are \textit{not} asymptotically globally stable. We provide general conditions on the coefficients $b,\s$ that lead to synchronization by noise. Hence, in these cases the inclusion of additive noise in \eqref{eq:SDE_intro} stabilizes the long-time dynamics.

As a model example, one may consider the multidimensional double-well potential with additive noise, that is
\begin{equation}\label{eq:dbl_well_intro}
   dX_t = \left(X_t - |X_t|^2X_t \right)dt + \s dW_t\quad\text{on }\R^d.
\end{equation}
In this case, for $\s = 0$ the long-time dynamics are not asymptotically globally stable, but the attractor is given by the closed unit ball $\bar B(0,1)$. We shall also analyze the associated point attractor, which consists of all invariant points, i.e. $S^{d-1}\cup\{0\}$, where $S^{d-1}$ is the $(d-1)$-dimensional unit sphere. It follows from the general conditions developed in this paper, that for $\s>0$ synchronization occurs, that is, the random attractor collapses into a single (random) point.

The paper is split into two main parts. In the first part, Section \ref{sec:synchr_polish}, we provide general and new sufficient conditions for synchronization by noise for RDS on separable metric spaces. In the second part, Section \ref{sec:synchr_SDE}, these general conditions are verified for classes of SDE of the type \eqref{eq:SDE_intro}, thus proving synchronization by noise for SDE with additive noise on $\R^d$.

The sufficient conditions for synchronization by noise developed in Section \ref{sec:synchronization} are essentially sharp in the case of SDE driven by additive noise, i.e.\ sufficient and necessary. On the other hand, their verification in applications may rely on stronger but easily checkable assumptions. In particular, we verify the general conditions under an {\em eventual monotonicity} condition on the drift $b$.  In particular, this yields synchronization for \eqref{eq:dbl_well_intro}. However, this eventual monotonicity condition is not necessary for synchronization. This issue is resolved, in a second step, in Section \ref{sec:weak_sync} by concentrating on a weaker concept of synchronization, so-called weak synchronization. Weak synchronization means that there is a minimal weak \textit{point} attractor $A$ consisting of a single random point. Our results on weak synchronization are particularly complete in the case of gradient-type SDE, i.e. for 
\begin{equation}\label{eq:grad_SDE_intro}
    dX_t=-\nabla V(X_t)dt+\sigma dW_t\quad\text{on } {\mathbb{R}}^d,
\end{equation}
with $V \in C^2({\mathbb{R}}^d,{\mathbb{R}})$, $\sigma >0$ and $b:=-\nabla V$ satisfying a one-sided Lipschitz condition (among other assumptions). In particular, no eventual monotonicity condition has to be assumed in this case.

In fact, the concept of weak synchronization turns out to be of independent interest with intriguing relations to strong mixing properties of the associated Markovian semigroup (cf.\ in particular Proposition \ref{prop:existence_weak_RA} below). 

The proof of weak synchronization is based on an analysis of the support properties of the statistical equilibrium, which leads us to (partial) generalizations of results developed in \cite{LJ87,KS04}. 

The main results for RDS on separable metric spaces are given in Theorem \ref{maintheorem} concerning synchronization and Theorem \ref{thm:weak_synchronization} concerning weak synchronization by noise. The main result on weak synchronization for gradient-type SDE is given in Theorem \ref{thm:weak_synchronization_gradient}. The results on general classes of SDE of the type \eqref{eq:SDE_intro} are given in Section \ref{sec:summary}.

\subsection{Comments on the existing literature} There are several distinct approaches to synchronization by noise to be found in the literature. We distinguish three main types of arguments (without aiming for completeness here): Order-preserving RDS, local stability and transitivity of the two point motion, perturbation techniques based on large deviation results. 

Synchronization by noise for order-preserving, strongly mixing RDS $\varphi$ has been analyzed, for example, in \cite{AC98,CS04,C02,CCLR07,CCK07,FGS15,G13-4} and rather general results on (weak) synchronization have been obtained. However, assuming $\varphi$ to be order-preserving is a significant restriction, leading to stringent assumptions on the drift $b$ for \eqref{eq:SDE_intro} in dimensions larger than one (cf.\ \cite{C02}). In particular, our model example \eqref{eq:dbl_well_intro} is covered for $d=1$ only.

In \cite{B91}, Baxendale proves synchronization for SDE on compact manifolds, assuming ergodicity, local stability, in the sense that the top Lyapunov exponent is supposed to be negative, and assuming transitivity of the two point motion (condition (4.1) in \cite{B91}). As compared to Baxendale's work, we focus on the case of RDS on not necessarily compact separable metric spaces $E$. In particular, it is one of the aims of this paper to provide conditions for synchronization by noise that are easily checkable for SDE with additive noise on $\R^d$. The resulting conditions are rather different and not easily comparable to those developed in \cite{B91}, which are quite specific to the compact case.

Another approach, based on large deviation techniques, has been introduced in \cite{MS88,MSS94,T08}. Besides several technical assumptions, assuming for \eqref{eq:SDE_intro} that $b$ has only finitely many fixed points and that $\s$ is small enough, these works prove synchronization by noise. Again, we note that the model example \eqref{eq:dbl_well_intro} is covered for $d=1$ only. In contrast, all examples treated in \cite{T08} are easily seen to be included in our results.

Synchronization by linear multiplicative noise has been investigated in \cite{ACW83,CR04}. For the related effect of synchronization in master-slave systems we refer to \cite{C10} and the references therein. Synchronization for discrete time random dynamical systems (iterated function systems) has also been investigated and the recent results are deep and
advanced, see \cite{H13, N14,KJR13,JK13} and references therein.

Synchronization has been advocated as a relevant property for certain applications. From the theoretical physics literature let us mention \cite{RSTA95,P96,PRKH02,KJR13}. In climate dynamics it has been mentioned as an indication of the possibility to reduce variability of predictions, see \cite{GCS08,CSG11,C13}. In neurophysiology, synchronous firing of neurons subject to the same input, which may be seen as a dynamical system driven by the same noise path but different initial conditions, is a phenomenon of interest, see \cite{SSV14} and the references therein. Finally, synchronization plays a role in Richardson-Romberg extrapolation numerical method, see \cite{LPP13}.

\subsection{Preliminaries and notation}\label{sec:prelim_not}

Let $(E,d)$ be a complete separable metric space with Borel $\sigma$-algebra $\EE$ and $\left(  \Omega,\F,\PP, \t\right)  $ be an {\em ergodic metric dynamical system}, i.e.\ 
$(\Omega,\F,\PP)$ is a (not necessarily complete) probability space and  $\t:=\left(  \theta_{t}\right)  _{t\in\R}$ 
is a group of jointly measurable maps on $\left(  \Omega,\F,\PP\right)$ with ergodic invariant measure $\PP$.

Further, 
let $\varphi: \R_+ \times \Omega \times E \rightarrow E$ be a {\em perfect cocycle}: i.e. $\varphi$ is measurable, $\varphi_{0}(  \omega,x)  =x$ and $\varphi_{t+s}\left(
\omega,x\right)  =\varphi_{t}\left(  \theta_{s}\omega,\varphi_{s}\left(
\omega,x\right)  \right)  $ for all $x\in E$, $t,s\geq0$, $\omega\in\Omega$. 
We will assume that $\varphi_{s}(\w,\cdot)$ is continuous for each $s \ge 0$ and $\omega \in \Omega$. 
The collection $(\Omega, \F,\PP,\t,\varphi)$ is then called a 
{\em random dynamical system} (in short: RDS), see \cite{A98} for a comprehensive treatment. 

By definition, $(\Omega, \F,\PP,\t,\varphi)$ is a {\em local} RDS if $(\Omega, \F,\PP,\t)$ is as above and $\varphi: \R_+\times \Omega \times \bar E \to \bar E$ is measurable, where $\bar E:= E \cup \{\partial\}$ and $\partial$ is some adjoined state with the following properties: 
$D:=\varphi^{-1}(E)\subseteq \R_+ \times \Omega \times E  $ and 
for each $\omega \in \Omega$ the set $D(\omega):=\{(t,x)\in \R_+\times E:(t,\omega,x)\in D\}$ is open, $(t,x) \in D(\omega)$ and $0\le s\le t$ imply $(s,x) \in D(\omega)$, 
$x \mapsto \varphi_t(\omega,x)$ is continuous at $x_0 \in E$ whenever 
$\varphi_t(\omega,x_0)\in E$, $\varphi_0(\omega,.)={\mathrm{Id}}$ and $\varphi$ has the perfect cocycle property (as above). Note that a local RDS is an RDS iff 
$D=\R_+ \times \Omega \times E$. 
Given a (local) RDS $(\Omega, \F,\PP,\t,\varphi)$ we may define the {\em skew-product} flow $\Theta$ on $\Omega\times \bar E$ by $\Theta_t(\omega,x)=(\t_t\omega,\varphi_t(\omega,x))$. 
In the following we will often omit the qualifier {\em local}. We say that a local RDS is {\em weakly complete} if $\varphi_t(\cdot,x) \in E$, $\PP$-a.s.\ for all $t\ge 0,x\in E$.

Since our main applications are RDS generated by SDE driven by Brownian motion, we will assume that the RDS $\varphi$ is suitably 
adapted to a filtration and is of white noise type. More precisely, we will assume that we have a family $\FF=(\F_{s,t})_{-\infty < s \le t < \infty}$ of sub$-\sigma$ algebras
of $\F$ such that $\F_{t,u}\subseteq \F_{s,v}$ whenever $s \le t \le u \le v$, 
$\theta_r^{-1}(\F_{s,t})=\F_{s+r,t+r}$ for all $r,s,t$ and $\F_{s,t}$ and $\F_{u,v}$ are independent whenever 
$s \le t \le u \le v$. For each $t \in \R$, let us denote the smallest $\sigma$-algebra containing all $\F_{s,t}$, $s \le t$ by $\F_t$ 
and the smallest $\sigma$-algebra containing all $\F_{t,u}$, $t \le u$ by $\F_{t,\infty}$. Note that for each $t \in \R$, the $\sigma$-algebras 
$\F_t$ and $\F_{t,\infty}$ are independent. We will further assume that $\varphi_{s}(\cdot,x)$ is $\F_{0,s}$-measurable for each $s \ge 0$. The collection $(\Omega, \F,\FF,\PP,\t,\varphi)$ is then called a {\em white noise (filtered) random dynamical system}.

An invariant measure for an RDS $\varphi$ is a probability measure on $\Omega\times E$ with marginal $\PP$ on $\Omega$ that is invariant under $\Theta_t$ for $t\ge 0$. 
For each probability measure $\mu$ on $\Omega\times E$ with marginal $\PP$ on $\Omega$ there is a unique disintegration $\omega\mapsto \mu_\omega$ and $\mu$ is an 
invariant measure for $\varphi$ iff $\varphi_t(\omega)\mu_\omega = \mu_{\t_t\omega}$ for all $t\ge 0$, almost all $\omega\in \Omega$. Here $\varphi_t(\omega)\mu_\omega$ denotes the push-forward of $\mu_\omega$ under $\varphi_t(\omega)$. An invariant measure $\mu_\omega$ is said to be a Markov measure, if $\omega\mapsto \mu_\omega$ is measurable with respect to the past $\F_0$.
In case of a weakly complete, white noise RDS $\varphi$ we may define the associated Markovian semigroup by
  $$ P_t f(x):=\E f(\varphi_t(\cdot,x)), $$
for $f$ being measurable, bounded. There is a one-to-one correspondence between invariant measures for $P_t$ and Markov invariant measures for $\varphi$: If $\rho$ is $P_t$-invariant, then for every sequence $t_k \to \infty$ the weak$^*$ limit
\begin{equation}\label{eq:inv_meas}
   \mu_\omega := \lim_{k\to\infty}\varphi_{t_k}(\t_{-{t_k}}\omega)\rho 
\end{equation}
exists $\PP$-a.s. The weak$^*$ limit $\mu_\omega$ does not depend on the sequence $t_k$, $\PP$-a.s.\ and defines a Markov invariant measure for  $\varphi$. Vice versa, $\rho:=\E\mu_\omega$ defines an invariant measure for $P_t$. Note that the proof of these facts given in \cite{C91} applies without change to local RDS (cf.\ also \cite{CDS11} and \cite{KS12}).

We say that a Markovian semigroup $P_t$ with invariant measure $\rho$ is \textit{strongly mixing} if 
 $$P_t f(x) \to \int_E f(y) d\rho(y)\quad \text{for } t\to\infty$$ 
for each continuous, bounded $f$ and all $x\in E$. Similarly, we say that an RDS $\varphi$ is strongly mixing if the law of $\varphi_t(\cdot,x)$ converges to $\rho$ for $t\to\infty$ for all $x\in E$.

As a notational convention, we let 
  $$B(x,r):=\{y\in E: d(x,y)<r\}$$
be the open ball of radius $r$ centered at $x$ and $\bar{B}(x,r)$ the respective closed ball. For a set $A\subseteq E$ we let  
  $$A^\e := \{y\in E: d(y,A)=\inf_{a\in A}d(y,a)<\e\}$$
and
 $$\diam(A):=\sup_{a,b\in A}d(a,b).$$

\begin{defn}
  A family $\{D(\w)\}_{\w \in \Omega}$ of non-empty subsets of $E$ is said to be
  \begin{enumerate}
  \item  a random closed (resp. compact) set if it is $\PP$-a.s.\ closed (resp. compact) and $\w \mapsto d(x,D(\w))$ is $\mathcal{F}$-measurable for each $x \in E$. In this case we also call $D$, $\mathcal{F}$-measurable.
  \item $\varphi$-invariant, if for all $t\ge0$
  		$$\varphi_t(\w,D(\w))=D(\t_t\w),$$
  		for almost all $\w\in \Omega$.
  \end{enumerate}
  
\end{defn}

Next, we recall the definition of a pullback attractor and a weak (random) attractor (cf. \cite{CF94,O99}). 
 
\begin{defn}
Let $(\Omega, \F,\PP,\t,\varphi)$ be an RDS. A random, compact set $A$ is called 
a {\em pullback attractor}, if
\begin{enumerate}
\item $A$ is $\varphi$-invariant, and 
\item for every compact set $B$ in $E$, we have
$$
\lim_{t \to \infty}\sup_{x \in B} d(\varphi_t(\theta_{-t}\omega,x),A(\omega))=0, \mbox{ almost surely}.
$$
\end{enumerate}
The map $A$ is called a {\em weak attractor}, if it satisfies the properties above with almost sure convergence replaced by 
convergence in probability in (2). It is called a {\em (weak) point attractor}, if it satisfies the properties above with compact sets $B$ replaced by single points in (2).

A (weak) point attractor is said to be {\em minimal} if it is contained in each (weak) point attractor.
\end{defn}

Clearly, every pullback attractor is a weak attractor but the converse is not true (see e.g.~\cite{S02} for examples).
\begin{lemma}\label{lem:uniqueness weak attractors}
  Weak attractors (and hence pullback attractors) are unique in the sense that if an RDS has two weak attractors, then they agree almost surely.
\end{lemma}
\begin{proof}
  Let $A,\tilde{A}$ be two weak random attractors. Since $\tilde{A}$ is a random compact set, by \cite[Proposition 3.15]{C02-2} for each $\e>0$ there is a compact, deterministic set $K_\e$ and such that
    $$\PP[\tilde{A}\subseteq K_\e]\ge 1-\e.$$
  Since $A$ weakly attracts compact sets, for all $\delta,\e>0$ there is a $t_0(\delta,\e)$ such that
    $$\PP[d(\varphi_{t}(\w,K_\e),A(\t_{t}\w))>\delta] \le \e,\quad   \forall t\ge t_0.$$
  Hence, also   
      $$\PP[d(\varphi_{t}(\w,\tilde{A}(\w)),A(\t_{t}\w))>\delta] \le 2\e,\quad   \forall t\ge t_0.$$
  By invariance $\varphi_{t}(\w,\tilde{A}(\w))=\tilde{A}(\t_{t}\w), \PP$-a.s..
  Thus,
   $$\PP[d(\tilde{A}(\w),A(\w))>\delta]=\PP[d(\tilde{A}(\t_{t}\w),A(\t_{t}\w))>\delta] \le 2\e,\quad \forall t \ge t_0.$$
  Since $\e$ is arbitrary we conclude
     $$\PP[d(\tilde{A}(\w),A(\w))>\delta]=0\quad  \forall \delta>0,$$
   which implies the claim.
\end{proof}

We point out that the previous lemma does not hold for weak point attractors (see e.g.~Example \ref{Beispiel}). If an RDS has a weak attractor $A$, then $A$ admits an $\F_0$-measurable version by Lemma \ref{lem:uniqueness weak attractors} and \cite[Corollary 4.5]{CDS09}, that is, there exists an  $\F_0$-measurable weak attractor $\tilde A$ such that $A=\tilde A$, $\PP$-a.s.~(note that we did not assume $\F_0$ to be complete).

When discussing (weak or pullback) attractors we will always assume that the underlying RDS is global. In contrast, we allow the RDS to be local when we discuss 
invariant measures and (weak) point attractors. The existence of an invariant measure does not guarantee that the RDS is global but it does impose some obvious constraints on the set $D$ 
in the definition of a local RDS. 

\section{\gray{(Weak) }Synchronization \blue{and weak synchronization} for RDS on \blue{complete} separable metric spaces}\label{sec:synchr_polish}
\subsection{Synchronization}\label{sec:synchronization}

In this section we introduce general sufficient conditions for synchronization by noise for RDS on \blue{complete }separable metric spaces. More precisely, we show that asymptotic stability (a local stability condition), swift transitivity (an irreducibility condition) and contraction on large sets imply synchronization by noise. If $E$ is  Heine-Borel, then asymptotic stability and contraction on large sets are also necessary conditions. Moreover, swift transitivity is satisfied by SDE of the type \eqref{eq:SDE_intro} with locally Lipschitz drift satisfying a one-sided Lipschitz condition (cf.\ Proposition \ref{prop:swift-SDE} below).

We can now define formally what we mean by synchronization for a given RDS $\varphi$.

\begin{defn}
We say that {\em synchronization} occurs if there is a weak attractor $A\left(  \omega\right)$ being a
singleton, for $\PP$-a.e. $\omega\in\Omega$.
\end{defn}

The problems of existence of a weak attractor and synchronization, i.e.\ it consisting of a single random point, are of quite different nature. Indeed, existence of weak attractors for SDE of the type \eqref{eq:SDE_intro} essentially relies on coercivity conditions for the drift $b$. For example, the existence of weak attractors for \eqref{eq:SDE_intro} under a mild coercivity condition has been shown in \cite{DS11}. In this work we shall concentrate on the problem of synchronization and thus assume the existence of a weak attractor.

We will now formulate sufficient conditions for synchronization to occur. 

\begin{defn}
\label{definition stability}Let $U\subset E$ be a (deterministic) non-empty open set. We
say that $\varphi$ is {\em asymptotically stable on $U$} if 
there exists a (deterministic) sequence $t_n \uparrow \infty$ such that 
\begin{equation}\label{assumption stability}
\PP\big( \lim_{n \to \infty}\diam(\varphi_{t_n}(.,U))=0\big)>0.
\end{equation}
\end{defn}

\begin{rem}[Necessity of asymptotic stability]\label{rem:nec_as}
Assume that synchronization holds and that there is at least one non-empty, open set $U\subseteq E$ that is attracted by $A(\omega)=\{a(\omega)\}$ (this is always true if $E$ is \blue{locally compact}), i.e.
            $$d(\varphi_t(\omega,U),A(\t_t \omega))\to 0\quad\text{for }t\to\infty,$$
        in probability. Then, $\varphi$ is asymptotically stable on $U$. Indeed: 
          \begin{align*}
            \diam(\varphi_t(\omega,U)) 
            &= \sup_{x,y\in U}d(\varphi_t(\omega,x),\varphi_t(\omega,y))\\
            &\le  \sup_{x,y\in U}d(\varphi_t(\omega,x),a(\t_t\omega))+d(a(\t_t\omega),\varphi_t(\omega,y))\\
            &\to 0\quad\text{for }t\to\infty
          \end{align*}
          in probability.
\end{rem}

Clearly, property \eqref{assumption stability} follows from the stronger assumption
\begin{equation}\label{other assumption stability}
\PP\left(  \lim_{t\rightarrow+\infty}\mathrm{diam}\left(  \varphi_{t}\left(\cdot,U\right)  \right)  =0\right)  >0,
\end{equation}
but there are a number of interesting cases in which \eqref{assumption stability} holds but \eqref{other assumption stability} does not (cf.\ also Remark \ref{eq:stable_mfd_and_as} below):
\begin{ex}\label{ex:as}
   We provide an example of an RDS $\psi$ satisfying asymptotic stability, i.e. \eqref{assumption stability}, but not satisfying \eqref{other assumption stability} regardless of the choice of $U$.
  Consider the one-dimensional SDE
      $$  dX_t=-X_tdt + dW_t$$
    with associated RDS $\varphi$. Obviously, $\varphi_t(\omega,x) - \varphi_t(\omega,y)=(x-y)e^{-t}$. Let now $t_n,x_n \uparrow \infty$ be such that 
    \begin{equation}\label{eq:as_comp}
      \PP\left( \sup_{t\in [t_{n-1},t_n]} \varphi_t(\cdot,x) \ge x_n\right) \to 1\quad\text{for }n\to\infty,
    \end{equation}
    for all $x\in\R$. We choose $f:\R\to\R$ smooth, strictly increasing with $\range(f)=\R$ such that
     $$f'(x)\ge n e^{t_n}\quad \forall x \in [x_n,x_{n+1}]$$
    and set $\psi_t(\omega,x):=f(\varphi_t(\omega,f^{-1}(x)))$. Let $y>x$. Then $$\psi_t(\omega,y)-\psi_t(\omega,x)\ge n(f^{-1}(y)-f^{-1}(x))$$
    if $\varphi_t(\omega,f^{-1}(x))\ge x_n$ and $t\in [t_{n-1},t_n]$. Due to \eqref{eq:as_comp} this happens i.o.\ $\PP$-a.s.. Hence, for all $y>x$ we have
      $$\limsup_{t\to\infty}|\psi_t(\cdot,x)-\psi_t(\cdot,y)|=\infty\quad\PP\text{-a.s.}$$
    and thus \eqref{other assumption stability} does not hold. In contrast, \eqref{assumption stability} is easily verified for $\psi$ \blue{and $U \subseteq \R$ bounded}.
\end{ex}

Let us first state Lemma \ref{Lemma general}, a very general and almost obvious criterion for synchronization.

\begin{lemma}\label{Lemma general}
  Let $\varphi$ be asymptotically stable on $U$ and $A$ be an $\mathcal{F}_0$-measurable, $\varphi$-invariant, random closed set with
  \begin{equation}
    \PP\left(  A\subset U\right)  >0 \label{assumption inclusion}.
  \end{equation}
  Then $A$ is a singleton $\PP$-a.s..
\end{lemma}
\begin{proof}

By property \eqref{assumption stability} there exists a sequence $t_n \uparrow \infty$ such that
$$
\PP\left(  \lim_{n\to \infty} \diam(\varphi_{t_n}(\cdot,U))=0\right)>0.
$$
Since $\{A \subset U\}$ is $\F_0$-measurable, $\{\lim_{n\to \infty} \diam(\varphi_{t_n}(\cdot,U))=0\}$ is $\F_{0,\infty}$-measurable and 
$\F_0$ and $\F_{0,\infty}$ are independent, we obtain 
$$
\PP\left(  \lim_{n\to \infty} \diam(\varphi_{t_n}(\cdot,A))=0\right)>0.
$$
In particular, since $\diam(\varphi_{t_n}(\cdot,A))$ has the same law as $\diam(A)$, we get 
$$
\PP\left( \diam(A)=0 \right)>0.
$$
We need to show that this probability is in fact 1. We observe that for each $t\ge 0$ we have
  $$ \{\diam(A(\t_{-t}\omega))=0 \} \subseteq \{\diam(A(\omega))=0 \}$$ up to a set of measure 0.
Since $\t_t$ is $\PP$ invariant these events have the same $\PP$-mass and thus coincide almost surely. Note that $ \{\diam(A(\t_{-t}\omega))=0 \}$ is $\F_{-t}$-measurable. 
Hence, $\{\diam(A(\omega))=0\}$ is measurable with respect 
$\cap_{t <0}\bar{\F}_t$ which is trivial by Kolmogorov's 0-1 law. Here, $\bar{\F}_t$ is the $\PP$-completion of $\F_t$. Therefore, we get  $\diam(A(\omega))=0$ almost surely and the proof of the lemma is complete.  
\end{proof}

In applications to SDE, assumption
(\ref{assumption stability}) will be a consequence of the property that the
top Lyapunov exponent $\lambda_{top}$ is negative and a regularity estimate, although being more general (cf.\ Section \ref{sec:asymptotic stability} below). Example \ref{ex:as} provides an RDS satisfying \eqref{assumption stability}, but the top Lyapunov exponent does not exist.

Let us come to {\color{black} assumption \eqref{assumption inclusion}} of Lemma \ref{Lemma general}. We can view it as an obvious consequence 
of the following condition.

\begin{defn}
We say that a random closed set $A$ has {\em full support} if
\begin{equation}
\PP\left(  A\subset U\right)  >0 \label{assumption full support}
\end{equation}
for every non-empty (deterministic) open set $U\subset E$.
\end{defn}

Let us give a sufficient condition for full support.

\begin{defn}
\label{definition swift}We say that $\varphi$ is {\em swift transitive} if, for
every (starting) ball $ B\left(  x,r\right)$ and every (arrival) point $y$,
there is a time $t>0$ such that
\[
\PP\left(  \varphi_{t}\left(\cdot,  B\left(  x,r\right)  \right)  \subset B\left(
y,2r\right)  \right)  >0.
\]
\end{defn}

\begin{defn}\label{small diameter}
  A random closed set $A$ is said to have {\em small diameter} if\begin{equation}
  \mathrm{ess}\inf\left\{  \mathrm{diam}(A\left(  \omega\right)  );\omega\in
  \Omega\right\}  =0. \label{inf diam}%
  \end{equation}
\end{defn}
Condition (\ref{inf diam}) means that%
\[
\PP\left(  \mathrm{diam}\left(  A\right)  <\varepsilon\right)  >0
\]
for every $\varepsilon>0$ and we have the following equivalence:
\begin{lemma}\label{small ball equivalence}
  Let $A$ be a closed random set. Then {\color{black} $A$ has small diameter}  iff for each $\e>0$ there is an $x_{0}\in E$ such that
  \[
  \PP\left(  A\subset B\left(  x_{0},\varepsilon\right)  \right)  >0.
  \]
\end{lemma}
\begin{proof} Assume that \blue{$A$ has small diameter} and consider a countable family of balls of the form $B\left(  x_{n}%
    ,\varepsilon\right)  $ where $\{x_{n},n \in \mathbb{N}\}$ is a dense countable set in $E$. We know
    that%
    \[
    \PP\left(  \mathrm{diam}\left(  A\right)  <\varepsilon\right)  >0.
    \]
    We have%
    \[
    \left\{  \mathrm{diam}\left(  A\right)  <\varepsilon\right\}  \subset\left\{
    A\subset B\left(  x_{n},\varepsilon\right)  \text{ for some }n\in
    \mathbb{N}\right\}
    \]
    hence
    \begin{align*}
      0
      &<\PP\left(  A\subset B\left(  x_{n},\varepsilon\right)  \text{ for some }n\in\mathbb{N}\right) \\
      &=\PP\left(\bigcup_{n\in\N}  \{A\subset  B\left(  x_{n},\varepsilon\right)\}\right) \\
      &\le \sum_{n\in\N}\PP\left(A\subset  B\left(  x_{n},\varepsilon\right)\right) 
    \end{align*}
    and thus $\PP\left(  A\subset B\left(  x_{n},\varepsilon\right)  \right)  >0$
    for some $n\in\mathbb{N}$, proving the claim. {\color{black}The reverse implication is obvious.}
\end{proof}

\begin{lemma}
\label{lemma inf diam}If $\varphi$ is swift transitive and $A$ is an $\mathcal{F}_0$ measurable, $\varphi$-invariant random closed set with small diameter, then $A$ has full support.
\end{lemma}
\begin{proof}
Let $U$ be a non-empty open set and $B\left(  y,R\right) \subset U$. By Lemma \ref{small ball equivalence} there is an $x_{0}\in E$ such that $\PP\left(  A\subset B\left(
x_{0},\frac{R}{2}\right)  \right)  >0$. By Definition \ref{definition swift}
with the starting ball $B\left(  x_{0},\frac{R}{2}\right)  $ and the arrival
point $y$,\ there is a time $t>0$ such that
\[
\PP\left(  \varphi_{t}\left(\cdot,  B\left(  x_{0},\frac{R}{2}\right)  \right)
\subset B\left(  y,R\right)  \right)  >0.
\]
By the independence of $\mathcal{F}_{0,t}$ and $\mathcal{F}_{0}$ and the fact
that $\varphi_{t}$ is $\mathcal{F}_{0,t}$-measurable and $A$ is $\mathcal{F}%
_{0}$-measurable, it follows that%
\[
\PP\left(  A\subset B\left(  x_{0},\frac{R}{2}\right)  ,\varphi_{t}\left(\cdot,
B\left(  x_{0},\frac{R}{2}\right)  \right)  \subset B\left(  y,R\right)
\right)  >0
\]
and thus
\[
\PP\left(  \varphi_{t}\left(\cdot,  A\right)  \subset B\left(  y,R\right)  \right)
>0.
\]
We have $\varphi_{t}\left(  \omega,A\left(  \omega\right)  \right)  =A\left(
\theta_{t}\omega\right)  $, hence $\PP\left(  A\left(  \theta_{t}\cdot\right)
\subset B\left(  y,R\right)  \right)  >0$. By $\theta_{t}$ invariance of $\PP$
we get%
\[
\PP\left(  A\subset B\left(  y,R\right)  \right)  >0
\]
hence $\PP\left(  A\subset U\right)  >0$. The proof is complete.
\end{proof}

The property of swift transitivity is generally true for SDE with additive noise and drift satisfying a local one-sided Lipschitz condition, see Section
\ref{sec:synchr_SDE}. Concerning \blue{the small diameter property}, it looks also very
general; we proceed to provide a sufficient condition.

The examples we have in mind which fulfill \blue{the small diameter property} have the
following features. With some (presumably very small) probability, their
attractors are driven to regions of strong contraction, where the size of the
attractor strictly decreases (cf. Section \ref{sec:eventual monotonicity} below for examples). Possibly this procedure has to be iterated,
until we reach a specified small value of the diameter. Let us formalize one
of these steps in a definition.

\begin{defn}
\label{Def contract}We say that $\varphi$ is {\em contracting on large sets} if for
every $R>0$, there is a ball $B\left(  y,R\right)  $ and a time $t>0$ such
that
\[
\PP\left(  \mathrm{diam}\left(  \varphi_{t}\left(\cdot,  B\left(  y,R\right)  \right)
\right) \le \frac{R}{4}\right)  >0.
\]
\end{defn}

{\color{black}
\begin{rem}The definition requires contraction, of some ball, for {\it every} radius $R$, not only for large $R$. However, contraction of some ball of small 
radius is a consequence of a suitable local stability assumption, similar to those we already assume. Hence the distinguished feature of this new condition 
is the fact that some ball of large radius is contracted. The name of the property has been chosen with this idea in mind, although it is not comprehensive 
of the full power of the definition. 
\end{rem}}

\begin{rem}[Necessity of contraction on large sets] Assume that synchronization holds and that $A(\omega)=\{a(\omega)\}$ weakly attracts all closed, bounded sets \blue{(which is always true if $E$ is Heine-Borel}\footnote{\blue{For example, every metric space $(E,d)$ that is both locally compact and $\sigma$-compact allows an equivalent metric $d'$ such that $(E,d')$ is Heine-Borel \cite{WJ87}.}}), then $\varphi$ is contracting on large sets. This follows as in Remark \ref{rem:nec_as}.
\end{rem}

Let us state the main abstract result of this section.

{\color{black}
\begin{thm}\label{maintheorem}
Assume that $\varphi$ is swift transitive. Then:
\begin{enumerate}
  \item[(1)] Assume that $\varphi$ is contracting on large
    sets. Then every $\mathcal{F}_0$-measurable, $\varphi$-invariant random compact set $A$ has small diameter.
  \item[(2)] Assume that $\varphi$ is asymptotically stable on some non-empty open set
  $U\subset E$. Let $A$ be an $\mathcal{F}_0$ measurable, $\varphi$-invariant random closed set with small diameter. Then $A$ is a singleton.
  In particular, if there is a weak attractor $A$ with small diameter, then synchronization occurs.

\end{enumerate}
\end{thm}}
\begin{proof}
(1): \textit{Step 1:} We start by proving the following claim: If
\[
\PP\left(  A\subset B\left(  x_{0},r_{0}\right)  \right)  >0
\]
for some $r_{0}>0$, $x_{0}\in E$ then
\[
\PP\left(  A\subset  B\left(  x_{1},\frac{2}{3} r_{0}\right)  \right)  >0
\]
for some $x_{1}\in E$.

Indeed: 
apply Definition \ref{Def contract} with $R=2r_{0}$: there is $y_{1}\in E$,
$t_{1}>0$ such that
\[
\PP\left(  \mathrm{diam}\left(  \varphi_{t_{1}}\left(\cdot,  B\left(  y_{1}%
,2r_{0}\right)  \right)  \right)  \leq\frac{r_{0}}{2}\right)  >0.
\]
For every $t_{0}>0$, since $\PP$ is invariant under $\theta_{t_{0}}$, we also have
\[
\PP\left(  \mathrm{diam}\left(  \varphi_{t_{1}}\left(  \theta_{t_{0}}%
\cdot,B\left(  y_{1},2r_{0}\right)  \right)  \right)  \leq\frac{r_{0}}%
{2}\right)  >0.
\]
Apply Definition \ref{definition swift} with the starting ball equal to
$B\left(  x_{0},r_{0}\right)  $ and the arrival point equal to $y_{1}$:\ there
is a time $t_{0}>0$ such that
\[
\PP\left(  \varphi_{t_{0}}\left(\cdot,  B\left(  x_{0},r_{0}\right)  \right)  \subset
B\left(  y_{1},2r_{0}\right)  \right)  >0.
\]
We have $\left\{  \varphi_{t_{0}}\left(\cdot,  B\left(  x_{0}
,r_{0}\right)  \right)  \subset B\left(  y_{1},2r_{0}\right)  \right\} \in \F_{0,t_{0}}$ 
and $$\left\{
\mathrm{diam}\left(  \varphi_{t_{1}}\left(  \theta_{t_{0}}\cdot,B\left(
y_{1},2r_{0}\right)  \right)  \right)  \leq\frac{r_{0}}{2}\right\}  \in \F_{t_{0},t_{0}+t_{1}},$$ since $\left\{  \mathrm{diam}
\left(  \varphi_{t_{1}}\left(  B\left(  y_{1},2r_{0}\right)  \right)  \right)
\leq\frac{r_{0}}{2}\right\}  \in \F_{0,t_{1}}$ and $\theta_{t_{0}%
}^{-1}\F_{0,t_{1}}= \F_{t_{0},t_{0}+t_{1}}$. Since $\F_{0,t_{0}}%
$ and $\F_{t_{0},t_{0}+t_{1}}$ are independent, and 
  $$\varphi_{t_{1}}\left(  \theta_{t_{0}}\omega,\varphi_{t_{0}}\left(
\omega,B\left(  x_{0},r_{0}\right)  \right)  \right)  =\varphi_{t_{1}+t_{0}%
}\left(  \omega,B\left(  x_{0},r_{0}\right)  \right)  $$
 we deduce
\[
\PP\left(  \mathrm{diam}\left(  \varphi_{t_{1}+t_{0}}\left(\cdot,  B\left(
x_{0},r_{0}\right)  \right)  \right)  \leq\frac{r_{0}}{2}\right)  >0.
\]
{\color{black}
This implies
\[
\PP\left(  \varphi_{t_{1}+t_{0}}\left(\cdot,  B\left(
x_{0},r_{0}\right)\right) \subseteq \bar{B}(\varphi_{t_{1}+t_{0}}\left(\cdot,  x_{0}\right),\frac{r_{0}}{2}) \right)   >0.
\]
Let $\{z_m\}_{m\in \N}$ be dense in $\R^d$. Then
\begin{align*}
  &\big\{\omega\in\Omega:  \varphi_{t_{1}+t_{0}}\left(\omega,  B\left(
  x_{0},r_{0}\right)\right) \subseteq \bar{B}(\varphi_{t_{1}+t_{0}}\left(\omega,  x_{0}\right),\frac{r_{0}}{2}) \big\} \\
  &\subseteq \big\{\omega\in\Omega:  \varphi_{t_{1}+t_{0}}\left(\omega,  B\left(
    x_{0},r_{0}\right)\right) \subseteq B(z_m,\frac{2}{3}r_{0}) \text{ for some }m\in\N \big\}.
\end{align*}
}
Hence there is an $x_{1}\in E$ such that
\[
\PP\left(  \varphi_{t_{1}+t_{0}}\left(\cdot,  B\left(  x_{0},r_{0}\right)  \right)
\subset B\left(  x_{1},\frac{2}{3}r_{0}\right)  \right)  >0.
\]
By the independence of $\mathcal{F}_{0,t_{1}+t_{0}}$ and $\mathcal{F}_{0}$ and
the fact that $\varphi_{t_{1}+t_{0}}$ is $\mathcal{F}_{0,t_{1}+t_{0}}%
$-measurable and $A$ is $\mathcal{F}_{0}$-measurable, it follows that
\[
\PP\left(  A\subset B\left(  x_{0},r_{0}\right)  ,\varphi_{t_{1}+t_{0}}\left(\cdot,
B\left(  x_{0},r_{0}\right)  \right)  \subset B\left(  x_{1},\frac{2}{3}r_0\right)  \right)  >0.
\]
Hence%
\[
\PP\left(  \varphi_{t_{1}+t_{0}}\left(\cdot,  A\right)  \subset B\left(  x_{1}%
,\frac{2}{3}r_0\right)  \right)  >0.
\]
We have $\varphi_{t_{1}+t_{0}}\left(  \omega,A\left(  \omega\right)  \right)
=A\left(  \theta_{t_{1}+t_{0}}\omega\right)  $, hence $\PP\left(  A\left(
\theta_{t_{1}+t_{0}}\cdot\right)  \subset B\left(  x_{1},\frac{2}{3}r_0 \right)  \right)  >0$. By $\theta_{t_{1}+t_{0}}$ invariance of $\PP$ we get%
\[
\PP\left(  A\subset B\left(  x_{1},\frac{2}{3}r_0\right)  \right)  >0.
\]

\textit{Step 2:}
The proof is now obvious. Since $A$ is a random compact set we can choose
$r_{0}>0$, $x_{0}\in E$ such that
\[
\PP\left(  A\subset B\left(  x_{0},r_{0}\right)  \right)  >0.
\]
Given any $\varepsilon>0$, we
may apply step one iteratively until we get
\[
\PP\left(  A\subset B\left(  x,\varepsilon\right)  \right)  >0
\]
for some $x\in E$ and the proof is complete.

(2): Simple consequence of \blue{Lemma \ref{lemma inf diam} and Lemma \ref{Lemma general}}.
\end{proof}
We finish this section with a simple example which illustrates the concepts introduced above.

\begin{ex}
Consider the one-dimensional SDE
$$
dX_t=X_t\,dW_t,\;X_0=x.
$$
where $W$ is standard Brownian motion. The RDS generated by the solution is given by
$$
\varphi_t(\w,x)=xe^{-\frac t2 + W_t(\w)}.
$$
$A(\omega)=\{0\}$ is the weak attractor of $\varphi$, so synchronization occurs. The RDS $\varphi$ is asymptotically stable on any bounded open 
set $U \subset \R$ and is contracting on large sets but $\varphi$ is not swift transitive. Lemma \ref{Lemma general} can be applied but 
Lemma \ref{lemma inf diam}   and Theorem \ref{maintheorem} cannot.
\end{ex}

\subsection{Weak synchronization}\label{sec:weak_sync}

{\color{black}
Although contraction on large sets is a necessary condition for synchronization, it is not always easy to check for SDE. In Section \ref{sec:synchr_SDE} below, for \eqref{eq:SDE_intro} we prove that $b$ being monotone on large sets implies contraction on large sets (cf. Proposition \ref{prop:swift-SDE}). However, monotonicity on large sets is not necessary for synchronization.

Therefore, in this section we investigate a weaker form of synchronization, so-called weak synchronization. The main improvement is that we are able to prove weak synchronization without assuming contraction on large sets, which in turn allows us to consider drifts $b$ not necessarily monotone on large sets. 

More precisely, for strongly mixing, white noise RDS we prove that the weak asymptotic stability condition \eqref{eq:ptw_stab} (a pointwise local stability condition), pointwise strong swift transitivity and a global pointwise stability condition imply weak synchronization. Again, weak asymptotic stability and the pointwise stability condition are also necessary for weak synchronization, while pointwise strong swift transitivity is easily checked for \eqref{eq:SDE_intro} under mild conditions as above. }

We will assume throughout this subsection that $\varphi$ is a {\em local}, white noise RDS.

\begin{defn}
We say that {\em weak synchronization} occurs if there is a minimal weak point attractor $A\left(  \omega\right)$ being a
singleton, for $\PP$-a.e. $\omega\in\Omega$.
\end{defn}

If there is a weak attractor $A$, then $A$ contains each minimal weak point attractor. In particular, synchronization implies weak synchronization. 

We now introduce a weaker concept of asymptotic stability. The point is, that asymptotic stability in the sense of Definition \ref{definition stability} is not necessary for weak synchronization, while the following concept of weak asymptotic stability obviously is:
\begin{defn}
\label{definition weak_stability}Let $U\subset E$ be a (deterministic) non-empty open set. We
say that $\varphi$ is {\em weakly asymptotically stable on $U$} if 
there exists a (deterministic) sequence $t_n \uparrow \infty$ and a set $\M\subseteq \Omega$ of positive $\PP$-measure, such that, for all $x,y \in U$
\begin{equation}\label{assumption weak_stability}
  1_\M(\cdot) d(\varphi_{t_n}(.,x),\varphi_{t_n}(.,y))\to 0\quad\text{for }n\to \infty,
\end{equation}
in probability.
\end{defn}

\begin{rem}
   If weak synchronization occurs, then weak asymptotic stability is satisfied with $U=E$, $\M=\Omega$ and every sequence $t_n \to \infty$ since, for all $x,y\in E$ we have
     $$d(\varphi_{t}(.,x),\varphi_{t}(.,y))\to 0\quad\text{for }t\to \infty,$$
   in probability.
\end{rem}

{\color{black}Assume that the Markov semigroup corresponding to $\varphi$ has an ergodic invariant measure $\rho$ with corresponding statistical equilibrium $\mu_\omega$ given by \eqref{eq:inv_meas}.}

Local stability in terms of weak asymptotic stability can be nicely captured in terms of the support of the statistical equilibrium $\mu_\omega$, i.e.\ if $\varphi$ is weakly asymptotically stable then the support has to consist of finitely many random points. For RDS with negative top Lyapunov exponent and on compact manifolds this goes back to \cite{LJ87}.

\begin{lemma}\label{lem:statistical equilibrium} 
\begin{enumerate}
   \item The statistical equilibrium $\mu_\omega$ is either discrete or diffuse. More precisely, either $\mu_\omega$ consists of finitely many atoms of the same mass $\mathbb{P}$-a.s., i.e. there is an $N\in \N$ and $\mathcal{F}_0$-measurable random variables $a_1,\dots,a_N$ such that
         $$ \mu_\omega= \frac{1}{N}\sum_{i=1}^N \d_{a_i(\omega)}$$
   or  $\mu_\omega$ does not have point masses $\PP$-a.s.. 
   \item  Assume that $\varphi$ is weakly asymptotically stable on $U$ with $\rho(U)>0$. Then $\mu_\omega$ is discrete.
\end{enumerate}
\end{lemma}

\begin{proof}
 The proof uses modified arguments from \cite{LJ87}. 
 
 {\color{black}
 (1): Assume that $\mu_\omega$ is not diffuse, that is, $\mu_\omega$ has point masses with positive probability. We let \gray{$Q$}\blue{$\mu$} be the probability measure on $\Omega \times E$ with marginal $\PP$ on $\Omega$ and disintegration $\mu_\omega$. Let $g(\omega,x):=\mu_\omega(\{x\})$. Then 
 \begin{align*}
   g(\Theta_t(\omega,x))
%   &=g(\t_t\omega,\varphi_t(\omega,x)) \\
   &=\mu_{\t_t\omega}(\{\varphi_t(\omega,x)\}) \\
%   &=\varphi_t(\omega)\mu_{\omega}(\{\varphi_t(\omega,x)\}) \\
   &=\mu_{\omega}(\varphi_t(\omega,\cdot)^{-1}\{\varphi_t(\omega,x)\}) \\
%   &\ge \mu_{\omega}(\{x\}) \\
   &\ge g(\omega,x),\quad \gray{Q}\blue{\mu}-\text{a.s..}
 \end{align*}
 Since $\Theta_t$ is \gray{$Q$}\blue{$\mu$} ergodic (cf.\ \cite{C85}), this implies that $g$ is constant \gray{$Q$}\blue{$\mu$}-a.s.. Hence, all point masses of $\mu_\omega$ have the same mass $m\in \R_+$, $\PP$-a.s.. Since $\mu_\omega$ is not diffuse, we have $m>0$. Hence, $\PP$-a.s.,
  \begin{align*}
    m 
    &= \int_E g(\omega,x)d\mu_\omega(x) \\
    &= \int_E \mu_\omega(\{x\}) d\mu_\omega(x) \\
    &= N(\omega) m^2,
  \end{align*}
  where $N(\omega)$ denotes the number of point masses of $\mu_\omega$. This implies $N(\omega) = \frac{1}{m}$, $\PP$-a.s., which finishes the proof.
 }

(2): Due to (1) we only have to show that $\mu_\omega$ has a point mass with positive probability.

   \blue{Let $\Delta := \{(x,x):\ x\in E\}\subseteq E\times E$ be the diagonal in $E\times E$ and }let $\psi:(E\times E)\setminus\Delta \to [0,\infty)$ be measurable such that $\psi(x,y)\to\infty$ for $d(x,y)\to 0$ and 
     $$ \E\int_{(E\times E)\setminus\Delta} \psi(x,y)d\mu_\omega(x)d\mu_\omega(y)<\infty.$$
   In order to prove the existence of such a function $\psi$, define the probability measure $\gray{m}\blue{\nu} := \E \mu_\cdot \otimes \mu_\cdot$ on $E\times E$. Then, $\gray{m}\blue{\nu}(\Delta^\e \setminus \Delta) \to 0$ for $\e \to 0$. Hence, we can choose $\e_k \downarrow 0$ with $\e_k \le \e_0 = 1$ such that $\gray{m}\blue{\nu}(\Delta^{\e_k} \setminus \Delta) \le e^{-k}$. We then set 
     $$\psi(x,y) = 
     \begin{cases}
        k ,& \quad \text{if } (x,y) \in \Delta^{\e_k} \setminus \Delta^{\e_{k+1}} \\
        0 ,& \quad \text{if } |x-y| \ge 1.
     \end{cases}$$
   Let $U$ be as in the assumption of weak asymptotic stability. By invariance of $\mu_\omega$ we have
   \begin{equation}\begin{split}\label{eq:LJ}
     &\E\int_{(E\times E)\setminus\Delta} \psi(x,y)d\mu_\omega(x)d\mu_\omega(y)\\
     &\ge \E\int_{(E\times E)\setminus\Delta}\psi(\varphi_t(\omega,x),\varphi_t(\omega,y))d\mu_\omega(x)d\mu_\omega(y)\\
     &\ge \E\int_{(E\times E)\setminus\Delta}1_{U}(x)1_{U}(y)\psi(\varphi_t(\omega,x),\varphi_t(\omega,y))d\mu_\omega(x)d\mu_\omega(y).
   \end{split}\end{equation}
   By weak asymptotic stability there is a set $\M\subseteq\Omega$ with positive $\PP$-measure and a sequence $t_n \to \infty$ such that, for all $x,y\in U$
     $$ 1_\M(\cdot)d(\varphi_{t_n}(\cdot,x),\varphi_{t_n}(\cdot,y))\to 0\quad \text{for }n\to \infty,$$
   in probability. We define $C(n,x,y,R):=\{\omega\in \Omega: \psi(\varphi_{t_n}(\omega,x),\varphi_{t_n}(\omega,y))\ge R\}$ and observe
        $$\liminf_{n\to\infty} \PP(C(n,x,y,R)) \ge \PP(\M), $$
   {\color{black}for all  $x,y\in U$.} From \eqref{eq:LJ} we obtain
      \begin{equation*}\begin{split}
         &\E\int_{(E\times E)\setminus\Delta} \psi(x,y)d\mu_\omega(x)d\mu_\omega(y)\\
         &\ge R \E\int_{(E\times E)\setminus\Delta}1_{U}(x)1_{U}(y)1_{C(n,x,y,R)}(\omega)d\mu_\omega(x)d\mu_\omega(y).
      \end{split}\end{equation*}
      Since $\mu_\omega$ is $\F_{0}$-measurable, $C(n,x,y,R)$ is $\F_{0,\infty}$-measurable and $\F_{0}$, $\F_{0,\infty}$ are independent, we conclude
      \begin{equation*}\begin{split}
         &\E\int_{(E\times E)\setminus\Delta}1_{U}(x)1_{U}(y)1_{C(n,x,y,R)}(\omega)d\mu_\omega(x)d\mu_\omega(y)\\
         &=\E\E\left[\int_{(E\times E)\setminus\Delta}1_{U}(x)1_{U}(y)1_{C(n,x,y,R)}(\omega)d\mu_\omega(x)d\mu_\omega(y)\Big|\F_0\right]\\
         &=\E\tilde\E\int_{(E\times E)\setminus\Delta}1_{U}(x)1_{U}(y)1_{C(n,x,y,R)}(\tilde\omega)d\mu_\omega(x)d\mu_\omega(y)\\
         &=\E\int_{(E\times E)\setminus\Delta}1_{U}(x)1_{U}(y)\PP[C(n,x,y,R)]d\mu_\omega(x)d\mu_\omega(y)
      \end{split}\end{equation*}
      Using this above, taking $\liminf_{n\to\infty}$ and using Fatou's Lemma yields
      \begin{equation*}\begin{split}
         &\E\int_{(E\times E)\setminus\Delta} \psi(x,y)d\mu_\omega(x)d\mu_\omega(y)\\
         &\ge R \E\int_{(E\times E)\setminus\Delta}1_{U}(x)1_{U}(y)\liminf_{n\to\infty}\PP[C(n,x,y,R)]d\mu_\omega(x)d\mu_\omega(y)\\
         &\ge\PP[\M] R \E\int_{(E\times E)\setminus\Delta}1_{U}(x)1_{U}(y)d\mu_\omega(x)d\mu_\omega(y)
      \end{split}\end{equation*}
      If $\mu_\omega$ has no point masses, then {\color{black}$(\mu_\omega\otimes \mu_\omega)(\Delta)=0$ and thus	
      \begin{align*}
        \E\int_{(E\times E)\setminus\Delta}1_{U}(x)1_{U}(y)d\mu_\omega(y)d\mu_\omega(x)
        &=\E\int_{E\times E}1_{U}(x)1_{U}(y)d\mu_\omega(y)d\mu_\omega(x)\\
        &=\E ( \mu_\omega(U)^2 ) \\
        &\ge \rho(U)^2 \\
        &>0.
      \end{align*}}
      Since $R>0$ is arbitrary we obtain a contradiction. This concludes the proof.
\end{proof}

Let 
$$
E_0:=\{x \in E: \lim_{t \to \infty}P_t(x,.)=\rho\},
$$ 
where $P_t(x,.)$ denotes the transition probability and convergence is to be understood in the weak$^*$ sense.
Note that if the support of $\mu_\omega$ is compact with strictly positive probability then it is compact with probability one.

\begin{pro}\label{prop:existence_weak_RA}
  \begin{enumerate}
    \item Assume that $A(\omega):={\mathrm{supp}}(\mu_\omega)$ is (almost surely) compact. Then $A$ is a weak point attractor of the set $E_0$. In particular, if $\varphi$ is strongly mixing then $A$ is a minimal weak point attractor.
    \item If $\varphi$ is strongly mixing and weakly asymptotically stable on $U$ with $\rho(U)>0$, then there is an $N\in \N$ and $\mathcal{F}_0$-measurable random variables $a_1,\dots,a_N$ such that
        $$A(\omega)={\mathrm{supp}}(\mu_\omega)=\{a_i(\w):i=1,\dots,N\}$$
    is a minimal weak point attractor.
  \end{enumerate}
\end{pro}
\begin{proof} %[Proof of Proposition \ref{prop:existence_weak_RA}]
Under a compact absorption assumption, Proposition \ref{prop:existence_weak_RA} (1) corresponds to \cite[Theorem 2.4]{KS04}. The more general setting treated here, however, requires a quite different proof.

(1):  We show that $A$ attracts each $x \in E_0$ in probability.

Fix $\varepsilon>0$. {\color{black}Let $x_1,x_2,...$ be countable dense in $E$ and define $B_i:=B(x_i,\varepsilon/3)$ and
$$
I(\omega):=\min\big\{i:\,A(\omega) \subseteq \bigcup_{j=1}^i B_i\big\}.
$$
Note that $I(\omega)$ is almost surely finite. Let 
$$
\tilde A(\omega):=\bigcup_{j \in J(\omega)} B_j,
$$
where 
$$
J(\omega):=\{j \le I(\omega):\,B_j \cap A(\omega)\neq \emptyset\}. 
$$
Then  $\tilde A(\omega)$ is a random bounded open set and 
$$
\beta(\omega):=\min\{\mu_\omega (B_j):\,j \in J(\omega)\}
$$
is measurable and almost surely positive.}

For $b>0$ and $n\in \N$, we define
$$
A_n^{(b)}(\omega)= \bigcup_{x \in D_b(n,\omega)} B\big( x,\frac \varepsilon 3\big),
$$ 
where $D_b(n,\omega)$ is the set of all $x \in E$ for which $(\varphi_n(\theta_{-n}\omega)\rho) \big( B\big(x,\frac \varepsilon 3\big) \big) \ge b$.
{\color{black} Note that $A_n^{(b)}$ \blue{is $\F_{-n,0}$-measurable} (unlike $A$ and $\tilde A$ which are only \blue{$\F_0$-measurable})}.

{\color{black} Since  $(\varphi_{n}(\theta_{-{n}}\omega)\rho) \to \mu_\omega$ weakly$^*$ for $n\to\infty$, $\PP$-a.s.\ and $\mu_\omega(A(\omega)^\frac{\e}{3})=1$ we have, $\PP$-a.s.,
   $$(\varphi_{n}(\theta_{-n}\omega)\rho)(A(\omega)^\frac{\e}{3})\ge 1-\frac{b}{2}$$ 
for $n$ large enough. Hence,  $\PP$-a.s.\ for $x\in D_b(n,\omega)$ we have $B(x,\frac{\e}{3})\cap A(\omega)^\frac{\e}{3} \ne \emptyset$ for $n$ large enough. Thus,
$$
\lim_{n \to \infty} \PP\big( A_{n}^{(b)}(\omega) \subseteq \blue{A(\omega)^{\varepsilon}}\big) =1.
$$
Further, we have
$$
\liminf_{n \to \infty} \PP\big(\tilde A \subseteq A_{n}^{(b)} \big) \ge \PP\big( \beta(\omega)>b\big).
$$
For given $\delta>0$, we find $b$ so small and $n_0$ so large that 
\begin{equation}\label{eq:cover_A}
  \PP\big( \tilde A \subseteq A_{n}^{(b)} \subseteq A^\varepsilon\big) \ge 1-\frac{\delta}{4},
\end{equation} 
for all $n \ge n_0$.

\blue{Since $\tilde A(\omega)\supseteq A(\omega)$ we have $\mu_\omega(\tilde A(\omega))=1$. Thus, since $\tilde A(\omega)$ is an open set, weak$^*$ convergence} of $(\varphi_{n}(\theta_{-{n}}\omega)\rho)$ to  $\mu_\omega$ \gray{thus }implies $(\varphi_{n}(\theta_{-{n}}\omega)\rho)(\tilde A(\omega))\ge 1-\frac{\d}{3}$ for all $n$ large enough $\PP$-a.s.. Hence, \eqref{eq:cover_A} yields that there exists some $n_1 \ge n_0$ such that 
$$
\PP\left( (\varphi_{n}(\t_{-n}\cdot )\rho)(A_{n}^{(b)}(\cdot))\ge 1 - \frac{\delta}{3}\right) \ge 1-\frac{\delta}{3}
$$
for all $n \ge n_1$.}

Fix $x \in E_0$. Then, for all $n \ge n_1$,
\begin{align*}
&\liminf_{s \to \infty} \PP \big(\varphi_{s+n}\big(\theta_{-(s+n)}\omega,x\big) \in A^{\varepsilon}\big) \\
&\ge \liminf_{s \to \infty} \PP \big(\varphi_{s+n}\big(\theta_{-(s+n)}\omega,x\big) \in A^{(b)}_{n}\big) -\frac \delta 4\\
&= \liminf_{s \to \infty} \PP \big(\varphi_{s}\big(\theta_{-(s\blue{+n})}\omega,x\big) \in \varphi_n^{-1}(\theta_{-n}\omega)(A_n^{(b)})\big)               -\frac \delta 4\\
&\ge \E\big((\varphi_n(\theta_{-n}\omega)\rho)\big(  A_n^{(b)} \big)\big)  -\frac \delta 4\\
&\ge \big(1-\frac{\delta}3\big)^2-\frac \delta 4 \ge 1-\delta,
\end{align*}
where we used the independence of $\F_{-n,0}$ and $\F_{-\infty,-n}$, \blue{the fact that $A_n^{(b)}$ is $\F_{-n,0}$-measurable }and the fact that $x \in E_0$ in the step from the third to the fourth line. Since $\delta>0$ and $\varepsilon>0$ are arbitrary, the claim follows.

Let now $\varphi$ be strongly mixing, i.e.\ $E_0=E$. Minimality of $A$ follows from the fact that every $\varphi$-invariant Markov measure is supported by every weak point attractor $A'$ (cf. \cite{C01}), i.e. $\mu_\w(A'(\w))=1$ a.s.. Hence, $A'\subseteq A$ a.s..

(2): Follows from Lemma \ref{lem:statistical equilibrium} and (1).
\end{proof}

By Proposition \ref{prop:existence_weak_RA}, without any assumption on $A$ having full support, asymptotic stability of $\varphi$ implies that the minimal weak point attractor consists of finitely many points. We note that if $E$ is connected, then this is true for a weak attractor iff synchronization occurs. Indeed, if $E$ is connected then so are weak attractors (which follows from the same proof as for \cite[Proposition 3.13]{CF94}). The following example shows that Proposition \ref{prop:existence_weak_RA} is not true for weak attractors.
\begin{ex}\label{Beispiel}
  Consider
    $$d\alpha_t = \cos(2\alpha_t)\circ dW^1_t+\sin(2\alpha_t)\circ dW^2_t$$
  on the one-dimensional sphere $S^1$. Then the weak attractor is the whole sphere $S^1$ while the minimal weak point attractor consists of two (antipodal) random points $\PP$-a.s.\ (cf. \cite[Remark 4.11]{B91}).
\end{ex}

\begin{defn}\label{definition swift_2} We say that $\varphi$ is {\em pointwise strongly swift transitive} if there is a time $t>0$ such that for
every $x^1,x^2\in E$ and every (arrival) point $y$,
\[
\PP\left(  \varphi_{t}\left(\cdot,  \{x^1,x^2\} \right)  \subset B\left(
y,2d(x^1,x^2)\right)  \right)  >0.
\]
\end{defn}

We obtain
\begin{thm}\label{thm:weak_synchronization}
  Assume that $\varphi$ has right-continuous trajectories, is strongly mixing, weakly asymptotically stable on $U$ with $\rho(U)>0$, pointwise strongly swift transitive and
  \begin{equation}\label{eq:ptw_stab}
     \liminf_{t\to\infty}d(\varphi_t(\omega,x),\varphi_t(\omega,y))=0\quad\text{a.s., }  \forall x,y\in E.
  \end{equation}
   Then, there is a minimal weak point attractor $A$ consisting of a single random point $a(\omega)$ and
    $$A(\omega)=\supp(\mu_\omega)=\{a(\omega)\}\quad  \PP\text{-a.s.,}$$
   i.e.\ weak synchronization holds.
\end{thm}
\begin{proof} %[Proof of Proposition \ref{thm:weak_synchronization}]
Since $\varphi$ is strongly mixing and weakly asymptotically stable on some non-empty open set with positive $\rho$-measure, by Proposition \ref{prop:existence_weak_RA} there are $\F_0$-measurable random variables $a_i(\omega)$, $i=1,\dots,N$ such that
 $$A(\omega):=\{a_i(\omega):i=1,\dots,N\}$$ 
is a minimal weak point attractor.

\textit{Step 1}: 
   By weak asymptotic stability there is an open set $U$, a sequence $t_n\to\infty$ and a $\delta>0$ such that for all $x,y\in U$ and all $\eta>0$
   \begin{equation}\label{eqn:weak_asymptotic_stability}
        \liminf_{n\to\infty}\PP\left(d(\varphi_{t_n}(\cdot,x),\varphi_{t_n}(\cdot,y)) \le \eta \right) \ge \d >0. 
   \end{equation}

Without loss of generality we may assume $U=B(\e,x_0)$ for some $x_0 \in E$, $\e>0$. 
   Let $x,y\in E$. 
   By assumption, 
   the stopping time
     $$ \tau^{x,y}_\e(\omega) := \inf\{t\ge0:\ d(\varphi_t(\omega,x),\varphi_t(\omega,y))\le \frac{\e}{4}\} $$
   is finite $\PP$-almost surely. 
   
   Let now $a(\omega),b(\omega)\in A(\omega)$ be two $\F_0$-measurable selections and let $\tau_\e(\omega):=\tau_\varepsilon^{a(\omega),b(\omega)}$, where $\tau_\varepsilon^{x,y}$ is defined as above. Due to independence of $\F_0$ and $\F_{0,\infty}$, $\tau_\e$ is finite a.s.. Right-continuity of the trajectories implies that there is a $\iota:\Omega\to \R_+\setminus\{0\}$ such that
   \begin{equation}\label{eq:hit_1}
      d\left(\varphi_{\tau_\e(\omega)+t}(\omega,a(\omega)),\varphi_{\tau_\e(\omega)+t}(\omega,b(\omega))\right)\le \frac{\e}{3}
   \end{equation}
   for all $t\in [0,\iota(\omega)]$, $\PP$-a.s.. Hence, there is a $\bar t_0 \ge 0$ such that
     $$\PP\left(d(\varphi_{\bar t_0}(\cdot,a(\cdot)),\varphi_{\bar t_0}(\cdot,b(\cdot)))\le \frac{\e}{3}\right)>0.$$
   Indeed: Assume not. Then
     $$\PP\left(d(\varphi_{t}(\cdot,a(\cdot)),\varphi_{t}(\cdot,b(\cdot)))> \frac{\e}{3},\ \text{for all }t\in \Q_+\right)=1,$$
   in contradiction to \eqref{eq:hit_1}.
   
   \textit{Step 2}: By pointwise strong swift transitivity and using that $\varphi$ is a white noise RDS
   there is a time $\bar t_1 \ge 0$ such that
        $$\PP(\varphi_{\bar t_0 +\bar t_1}(\cdot,\{a(\cdot),b(\cdot)\})\subseteq U)>0.$$
   Again using that $\varphi$ is a white noise RDS
   we conclude
      \begin{equation}\label{eqn:pos_prob}
         \liminf_{n\to\infty}\PP\big(d\left(\varphi_{\bar t_0+\bar t_1+t_n}(\cdot,a(\cdot)),\varphi_{\bar t_0+\bar t_1+t_n}(\cdot,b(\cdot))\right)\le \eta\big)\ge\frac{\d}{2}>0.
      \end{equation}

\textit{Step 3}: Assume $A(\omega)$ is not a singleton $\PP$-a.s.. Then
\begin{equation}\label{eq:F_pos}
  F(\omega):=\min_{i,j=1,\dots,N,\ i\ne j}d(a_i(\omega),a_j(\omega)) >0,
\end{equation}
$\PP$-a.s.. 
Moreover, since $\varphi_t(\omega,A(\omega))=A(\t_t\omega)$ we have
\begin{align*}
    F(\t_t\omega)
    &=\min_{i,j=1,\dots,N,\ i\ne j}d(a_i(\t_t\omega),a_j(\t_t\omega))\\
    &=\min_{i,j=1,\dots,N,\ i\ne j}d(\varphi_t(\omega,a_i(\omega)),\varphi_t(\omega,a_j(\omega)))\\
    &\le d(\varphi_t(\omega,a_1(\omega)),\varphi_t(\omega,a_2(\omega))).
\end{align*}
Hence, for all $\eta>0$
\begin{align*}
    \PP(F(\cdot)\le \eta)
    &= \PP(F(\t_{\bar t_0+\bar t_1+t_n}\cdot)\le \eta) \\
    & \ge \PP\big(d(\varphi_{\bar t_0+\bar t_1+t_n}(\cdot,a_1(\cdot)),\varphi_{\bar t_0+\bar t_1+t_n}(\cdot,a_2(\cdot)))\le \eta\big).
\end{align*}
Taking $\liminf_{n\to\infty}$ and using \eqref{eqn:pos_prob} we conclude:
\begin{align}\label{eqn:stopped_pos_prob}
            \PP(F(\cdot)\le \eta)\ge \frac{\d}{2}>0,
\end{align}
for all $\eta$ in contradiction to \eqref{eq:F_pos}.
\end{proof}

{\color{black}As we will observe in Section \ref{sec:weak_synchr_SDE} below, condition \eqref{eq:ptw_stab} is satisfied for a large class of gradient-type SDE.}

\section{\gray{(Weak) }Synchronization \blue{and weak synchronization} for SDE}\label{sec:synchr_SDE}

In this section we provide classes of SDE satisfying asymptotic stability, swift transitivity and contraction on large sets and, thus, synchronization by noise. 

The section is divided into four parts. In the first part we will focus on asymptotic stability. We first develop a local stable manifold theorem for general, differentiable RDS and prove that a negative top Lyapunov exponent plus a regularity condition lead to asymptotic stability. We then provide sufficient conditions for SDE to have a negative top Lyapunov exponent and to satisfy this regularity condition. In the second part we will prove swift transitivity and contraction on large sets for SDE with additive noise. The third part concentrates on gradient-type SDE, proving weak synchronization under weak assumptions. Finally, these results are summarized and applied to SDE in the fourth part.

In this section we consider finite dimensional SDE driven by $d$-dimensional Brownian motion, i.e.
\begin{equation}\label{eq:ex_SDE}
dX_{t}=b\left(  X_{t}\right)  dt+\sigma\gray{\left(  X_{t}\right)}  dW_{t},
\end{equation}  
with $\s\blue{>0}$\gray{being Lipschitz continuous}, $b$ locally Lipschitz continuous and \gray{$b$ }satisfying the following one-sided Lipschitz condition
\begin{equation}\label{eq:b_mon}
  (b(x)-b(y),x-y)\le \lambda |x-y|^2,
\end{equation}
for all $x,y\in\R^d$ and some $\lambda >0$. 

By \cite{DS11} there is a white noise RDS $\varphi$ associated to \eqref{eq:ex_SDE}, with respect to the canonical setup: 
The space $\Omega$ is $C\left(  \mathbb{R};\mathbb{R}%
^{d}\right)  $, $\mathcal{F}$ is the (not completed) Borel $\sigma$-field, $\PP$ is the
two-sided Wiener measure, $\mathcal{F}_{s,t}$ is the $\sigma$-algebra generated by $W_{u}-W_{v}$ for
$s\leq v\leq u\leq t$, where $W_{s}:\Omega\rightarrow\mathbb{R}^{d}$ is
defined as $W_{s}\left(  \omega\right)  =\omega\left(  s\right)  $,
and $\theta_{t}$ is the shift%
\[
\left(  \theta_{t}\omega\right)  \left(  s\right)  =\omega\left(  s+t\right)
-\omega\left(  t\right)
\]
which is ergodic.

\subsection{Asymptotic stability and top Lyapunov exponent}\label{sec:asymptotic stability}
In this section we provide sufficient conditions for asymptotic stability for certain diffusions. We start by considering general RDS and proving that a negative top Lyapunov exponent implies asymptotic stability. Then we provide sufficient conditions for SDE to have negative top Lyapunov exponents.

\subsubsection{A time-discrete, local stable manifold theorem and asymptotic stability}
Let $\varphi$ be a white-noise RDS on $\R^d$ with respect to an ergodic metric dynamical system $(\Omega,\PP,\theta)$ and let $P_t$ be the associated Markovian semigroup. In this section we will introduce the associated Lyapunov spectrum under appropriate assumptions on $\varphi$ and provide a local stable manifold theorem for discrete time and negative top Lyapunov exponent. We then prove that this implies asymptotic stability.

\begin{lemma}\label{lem:stable_mfd}
   Assume that $\varphi_t(\omega,\cdot)\in C_{\mathrm{loc}}^{1,\d}$ for some $\d\in(0,1)$ and all $t\ge 0$. Further assume that {\color{black}$P_1$} has an ergodic invariant measure $\rho$ such that
   \begin{equation}\label{eq:integrability}
     \E \int_{\R^d} \log^+ \|D\varphi_1(\omega,x)\|d\rho(x) <\infty. 
   \end{equation}
   Then 
   \begin{enumerate} 
      \item[(1)] 
      There are constants $\l_N < \dots < \l_1$ (the Lyapunov spectrum) such that 
          $$ \lim_{n\to\infty}\frac{1}{n} \log |D\varphi_n(\omega,x)v| \in  \{\lambda_i\}_{i=1}^N,$$
       for all $v \in \R^d\setminus \{0\}$ and $\PP\otimes\rho$-a.a. $(\omega,x)\in \Omega\times\R^d$.
   \end{enumerate}
   We define the top Lyapunov exponent by $ \lambda_{top} := \lambda_1$. {\color{black} Assume
   \begin{equation}\label{eq:integrability_2}
     \E \int_{\R^d} \log^+(\|\varphi_1(\omega,\cdot+x)-\varphi_1(\omega,x)\|_{C^{1,\d}(\bar{B}(1,0))})d\rho(x)<\infty
   \end{equation}}
   and $\lambda_{top}<0$. Then
   \begin{enumerate} 
       \item[(2)] For every $\e\in (\l_{top},0)$ there is a measurable map $\b:\Omega\times \R^d\to \R_+\setminus\{0\}$ such that for $\rho$-a.a. $x \in \R^d$ \begin{align*} 
             {\S}(\omega,x) 
              &:=\{y\in \R^d:\ |\varphi_n(\omega,y)-\varphi_n(\omega,x)|\le \b(\omega,x)e^{\e n},\ \forall n\in\N\}
         \end{align*}
         is an open neighborhood of $x$, $\PP$-a.s..
   \end{enumerate}          
\end{lemma}
\begin{proof} (1): The introduction of the Lyapunov spectrum and the time-discrete stable manifold theorem will be based on \cite{R79}. In order to do so, we need to rewrite the dynamics in an appropriate form. This essentially follows the setup put forward in  \cite{MS99}.

  We define the following extension of the probability space (cf. e.g. \cite[p.\ 626 and Corollary 3.1.1, Remark (iii)]{MS99}): Let $M:=\Omega\times \R^d,\ \tilde{\F}:=\F_{0,\infty}\otimes\B(\R^d),\  \mu=\PP\otimes\rho|_{\tilde \F}$ and $\tau: M\to M$ be defined by 
    $$\tau(m) :=(\t_1 \omega,\varphi_1(\omega,x))\quad \text{for } m=(\omega,x)\in M.$$
  By \cite{C85} $\tau$ is ergodic. We then obtain a perfect (time-discrete) cocycle on $(M,\mu,\tau)$ by
    $$Z_n(m,y) := \varphi_n(\omega,y+x)-\varphi_n(\omega,x)\quad \text{for } m=(\omega,x)\in M.$$
  Note that $Z_n(m,0)=0$. We further set $F_{m}(y):=Z_1(m,y)$,
      $$F^n_{m} := F_{\tau^{n-1}m}\circ \dots \circ F_{m}$$
   and observe $F^n_{m}=Z_n(m,\cdot)$. 
  Obviously, $F_{m} \in C^{1,\d}$ and 
   $$T_m = T(m):=DF_{m}(0)=D\varphi_1(\omega,x)\quad \text{for }m=(\omega,x)\in M.$$
  Setting $T^n_m := T_{\tau^{n-1}m}\circ \dots \circ T_{m}$ we have $T^n_m=DF^n_{m}(0)=D\varphi_n(\omega,x)$ for $m=(\omega,x)\in M.$ 
  By assumption
    $$ \int \log^+ \|T(m)\|d\mu(m) = \E \int_{\R^d} \log^+ \|D\varphi_1(\omega,x)\|d\rho(x) <\infty.$$
  Since $\mu$ is ergodic, by the multiplicative ergodic theorem \cite[Theorem 1.6]{R79}, there are constants $\{\lambda_i\}_{i=1}^N$ (the Lyapunov spectrum) such that
    $$ \lim_{n\to\infty}\frac{1}{n} \log |T^n_m v| \in  \{\lambda_i\}_{i=1}^N$$
  and the limit exists for $\mu$-a.a. $m\in M$. Since $T^n_m v=D\varphi_n(\omega,x)v$ this finishes the proof.
  
  (2): By \cite[Theorem 5.1]{R79} (a), there are measurable maps $\b>\a>0$ such that, for a.a.\ $m=(\omega,x)\in M$,
    \begin{align*}
      &\{y\in B(0,\a(m)):\ |F_m^n y|\le \b(m)e^{n\e},\ \forall n\in\N\}\\
      &=\{y\in B(0,\a(\omega,x)):\ |\varphi_n(\omega,y+x)-\varphi_n(\omega,x)|\le \b(\omega,x)e^{n\e},\ \forall n\in\N\}
    \end{align*}
    is an open neighborhood of $0\in \R^d$, which implies (2).
\end{proof}

\begin{rem}\label{eq:stable_mfd_and_as}
    In contrast to the time-continuous local stable manifold theorem developed in \cite{MS99}, Lemma \ref{lem:stable_mfd} (2) only yields local stability along the natural numbers $n\in \N$. On the other hand, the assumptions of \cite{MS99} do not cover our model example of a double-well potential. At this point we make use of the weaker form, \blue{as compared to \eqref{other assumption stability}}, of asymptotic stability introduced in Definition \ref{definition stability}. In fact, as we will see below, Lemma \ref{lem:stable_mfd} (2) will be sufficient to deduce asymptotic stability, which significantly simplifies the proof of asymptotic stability in cases for which no local stable manifold theorem has (yet) been established.
\end{rem}

From Lemma \ref{lem:stable_mfd} (2) we obtain the existence of random neighborhoods of points, that are contracted under the (time-discrete) flow. The following Lemma clarifies the relation to asymptotic stability in the sense of Definition \ref{definition stability}. 
 
 \begin{lemma} \label{Lemma stability}
   Let $U_1$ be a random, non-empty, open set and assume that there is a sequence $t_n \to \infty$ such that
   \begin{equation}
   \PP\left(  \lim_{n\to\infty}\mathrm{diam}\left(  \varphi_{t_n}\left(\cdot,
   U_1(\cdot)  \right)  \right)  =0\right)  >0. \label{ineq 3}%
   \end{equation}
   Then there is a (deterministic) non-empty, open set $U$ such that
    $$\PP\left(  \lim_{n\to\infty}\mathrm{diam}\left(  \varphi_{t_n}\left(\cdot,U  \right)  \right)  =0\right)  >0.$$
   In particular, $\varphi$ is asymptotically stable in the sense of Definition \ref{definition stability}.
 \end{lemma}
\begin{proof}
Consider the countable family of balls of the form $B\left(  x_{m}%
,r_{m}\right)  $ where $(x_{m},r_m)$ is an enumeration of pairs of points $x_m$ of
$\mathbb{R}^{d}$ with rational coordinates and positive rational radii $r_{m}$. We have%
\[
\left\{  \lim_{n\to\infty}\mathrm{diam}\left(  \varphi_{t_n}\left(\cdot,
 U_1(\cdot)  \right)  \right)  =0\right\}  \subset\left\{
\lim_{n\to\infty}\mathrm{diam}\left(  \varphi_{t_n}\left(\cdot,  B\left(
x_{m},r_{m}\right)  \right)  \right)  =0,\text{for some }m\in\mathbb{N}%
\right\}.
\]
Hence, there exists $m\in\mathbb{N}$ such that
\[
\PP\left(  \lim_{n\to\infty}\mathrm{diam}\left(  \varphi_{t_n}\left(\cdot,
B\left(  x_{m},r_{m}\right)  \right)  \right)  =0\right)  >0.
\]
The ball $B\left(  x_{m},r_{m}\right)  $ is the set $U$ of the definition of
asymptotic stability. The proof is complete.
\end{proof}

As immediate consequence we obtain
\begin{cor}\label{cor stability}
  Let $\varphi$ be as in Lemma \ref{lem:stable_mfd} and assume \gray{\eqref{eq:integrability}, }\eqref{eq:integrability_2}, $\l_{top}<0$. Then $\varphi$ is asymptotically stable in the sense of Definition \ref{definition stability}.   
\end{cor}

\subsubsection{Examples with negative top Lyapunov exponent}\label{sec:add_noise}

In this section we provide three main classes of SDE for which we prove the top Lyapunov exponent to be negative. The first class of examples will be SDE with eventually monotone drifts and large noise. The second class consists of SDE with gradient structure and small noise. The third class are gradient-type SDE with certain symmetric potentials and all noise intensities.

Consider the following SDE with additive noise%
\begin{equation}\label{eq:lyapunov_add_SDE}
   dX_{t}=b\left(  X_{t}\right)  dt+\s dW_{t}\quad\text{on } \mathbb{R}^{d},
\end{equation}
where $\s>0$, $b \in C^{1,\d}_{loc}(\R^d)$ for some $\d\in(0,1)$ and $b$ satisfies \eqref{eq:b_mon}. Hence, there is a corresponding white noise RDS $\varphi$ with\footnote{\blue{Note that we have $\varphi_t(\omega,\cdot)\in C^{1,\d}_{loc}$ and not only $\varphi_t(\omega,\cdot)\in C^{1,\beta}_{loc}$ for every $\b<\delta$ due to the additive noise in \eqref{eq:lyapunov_add_SDE}. This may be seen by considering the transformation $\tilde \varphi = \varphi - \sigma W$ and studying the regularity of the corresponding pathwise ODE.}} $\varphi_t(\omega,\cdot)\in C^{1,\d}_{loc}$ and $D\varphi_t(\omega,x)$ satisfies the equation%
\[
\frac{d}{dt}D\varphi_{t}\left(  \omega,x\right)  =Db\left( \varphi_{t}\left(  \omega,x\right) \right)  D\varphi_{t}\left(  \omega,x\right)  ,\qquad D\varphi_{0}\left(
\omega,x\right)  ={\mathrm{Id}}.
\]
{\color{black}In particular, by \eqref{eq:b_mon}, given any $v\in\mathbb{R}^{d}\backslash\left\{  0\right\} $,%
\begin{equation}\label{eq:der_eqn}\begin{split}
\frac{d}{dt}\left\vert D\varphi_{t}\left(  \omega,x\right)  v\right\vert ^{2}
 &=2\left(  Db\left(\varphi_{t}\left(  \omega,x\right)   \right)  D\varphi_{t}\left(
\omega,x\right)  v,D\varphi_{t}\left(  \omega,x\right)  v\right)\\
 &\le 2\l|D\varphi_{t}\left(  \omega,x\right)  v|^2.
\end{split}\end{equation}
Gronwall's inequality yields
\begin{equation}\label{eq:der_bound}
  \| D\varphi_{t}\left(  \omega,x\right)  \| \le e^{\l t}. 
\end{equation}}

\begin{defn}\label{def:eventual strict monotonicity} A vector field $b:\R^d\to\R^d$ is said to be {\em eventually strictly monotone} if there exists an $R>0$ such that
  \[
  ( b\left(  x\right)  -b\left(  y\right)  ,x-y)
  \leq-\lambda_{1}\left\vert x-y\right\vert ^{2}\qquad\text{for all }\left\vert
  x\right\vert ,\left\vert y\right\vert >R
  \]
  for some $\lambda_{1}>0$.
\end{defn}

{\color{black} Assume that $b$ is eventually strictly monotone \blue{and $b \in C^{1,\d}_{loc}(\R^d)$ for some $\d\in(0,1)$}. 
Clearly, this implies \eqref{eq:b_mon} as well as the existence of some $c>0$, $C>0$ such that
\begin{equation}\label{eqn:assumptions for existence of flow}
     ( b\left(  x\right)  ,x) \leq-c\left\vert
     x\right\vert ^{2}+C   \qquad\text{for
     all }x\in \R^d.
\end{equation}
Indeed, \blue{let $D:=\sup_{|x|\le R}(b(x),x)$ and for $|x|>R$ let $y=\frac{Rx}{|x|}$. Then
\begin{align*}
  (b(x),x)
  &=(b(x)-b(y),x)+(b(y),x)\\
  &=\left(\frac{|x|}{|x|-R}\right)(b(x)-b(y),x-y)+(b(y),x)\\
  &\le -\lambda_1\left(\frac{|x|}{|x|-R}\right)|x-y|^2+\frac{|x|}{R}D\\
  &=-\lambda_1|x|(|x|-R)+\frac{D}{R}|x|\\
  &=-\lambda_1|x|^2+(\lambda_1 R+\frac{D}{R}) |x|.
\end{align*}}
\gray{First note that eventual monotonicity is equivalent to 
\[
(Db(x)h,h)\le-\l_{1}|h|^{2}\quad\forall h\in\R^{d},\,|x|>R.
\]
For $|x|>R$ we obtain that
\begin{align*}
(b(x),x) & =(b(x)-b(0),x)+(b(0),x)\\
 & =(\int_{0}^{1}Db(\tau x)xd\tau,x)+(b(0),x)\\
 & =\int_{0}^{\frac{R}{|x|}}(Db(\tau x)x,x)d\tau+\int_{\frac{R}{|x|}}^{1}(Db(\tau x)x,x)d\tau+(b(0),x)\\
 & \le\frac{R}{|x|}\|Db\|_{C(B(0,R))}|x|^{2}-\l_{1}|x|^{2}(1-\frac{R}{|x|})+(b(0),x)\\
 & \le(R\|Db\|_{C(B(0,R))}+|b(0)|)|x|-\l_{1}|x|^{2}(1-\frac{R}{|x|}).
\end{align*}}
Hence, \gray{for $|x|>2R$ and}using Young's inequality we obtain that
\begin{align*}
(b(x),x) & \le-\frac{\l_{1}}{2}|x|^{2}+C,\quad\forall |x|>R,
\end{align*}
which implies \eqref{eqn:assumptions for existence of flow} by local boundedness of $b$.

By \cite[Theorem 4.3]{K12} we obtain that $\varphi$ is strongly mixing with invariant probability measure $\rho$. Due to \eqref{eq:der_bound} this implies \eqref{eq:integrability}. Hence, an application of Lemma \ref{lem:stable_mfd} implies the existence of a corresponding (deterministic) Lyapunov spectrum with top Lyapunov exponent $\l_{top}$.}

\begin{ex}\label{ex:lyapunov_large_noise}
   Let $b\in C^{1,\d}_{\mathrm{loc}}$ for some $\d\in(0,1)$ {\color{black} be eventually strictly monotone} and consider the SDE
   \[
   \dd X_t=b(X_t)\,\dd t+\sigma\,\dd W_t\quad \text{on } \R^d
   \]
   with $\sigma>0$.
   If  $\sigma$ is large enough, then $\lambda_{top}<0$. 
\end{ex}
\begin{proof}
\textit{Step 1}: By \eqref{eq:der_eqn}, for any $v\in\mathbb{R}^{d}\backslash\left\{  0\right\}  $,%
\begin{align*}
\frac{d}{dt}\left\vert D\varphi_{t}\left(  \omega,x\right)  v\right\vert ^{2}
& =2\left(  Db\left( \varphi_{t}\left(  \omega,x\right)   \right)  D\varphi_{t}\left(
\omega,x\right)  v,D\varphi_{t}\left(  \omega,x\right)  v\right)  \\
& =2\left(  Db\left(  \varphi_{t}\left(  \omega,x\right)   \right)  r_{t}\left(  \omega
,x,v\right)  ,r_{t}\left(  \omega,x,v\right)  \right)  \left\vert D\varphi
_{t}\left(  \omega,x\right)  v\right\vert ^{2}%
\end{align*}
where $r_{t}\left(  \omega,x,v\right)  =\frac{D\varphi_{t}\left(
\omega,x\right)  v}{\left\vert D\varphi_{t}\left(  \omega,x\right)
v\right\vert }$. Hence,
\[
\left\vert D\varphi_{t}\left(  \omega,x\right)  v\right\vert ^{2}=\left\vert
v\right\vert ^{2} e^{  2\int_{0}^{t}\left(  Db\left(  \varphi_{s}\left(  \omega,x\right) 
\right)  r_{s}\left(  \omega,x,v\right)  ,r_{s}\left(  \omega,x,v\right)
\right)  ds}  .
\]
Recall that there exists a $v\in\mathbb{R}^{d}\backslash\left\{  0\right\}  $
such that
\[
\lambda_{top}=\lim_{n\rightarrow\infty}\frac{1}{n}\log\left\vert D\varphi
_{n}\left(  \omega,x\right)  v\right\vert .
\]
Hence,
\[
\lambda_{top}=\lim_{n\rightarrow\infty}\frac{1}{n}\int_{0}^{n}\left(
Db\left(  \varphi_{s}\left(  \omega,x\right)   \right)  r_{s}\left(  \omega,x,v\right)
,r_{s}\left(  \omega,x,v\right)  \right)  ds.
\]
With%
\begin{align*}
\lambda^{+}\left(  x\right)    & :=\max_{\left\vert r\right\vert =1}\left(
Db\left(  x\right)  r,r\right)  \\
\lambda^{-}\left(  x\right)    & :=\min_{\left\vert r\right\vert =1}\left(
Db\left(  x\right)  r,r\right)
\end{align*}
   we thus have
   \[
   \limsup_{n\to\infty}\frac{1}{n}\int_{0}^{n}\lambda^{-}(\varphi_{s}(\omega,x))ds\le\lambda_{top}\le\liminf_{n\to\infty}\frac{1}{n}\int_{0}^{n}\lambda^{+}(\varphi_{s}(\omega, x))ds.
   \]
Since $b$ \blue{satisfies \eqref{eq:b_mon}}\gray{is eventually strictly monotone}, we have that $\l^+(x)\le C$ for all $x\in \R^d$ and some constant $C>0$. Ergodicity and monotone convergence then yield
   \begin{equation}\label{eqn:top estimate}
       \lambda_{top}\le\int_{\R^d}\lambda^{+}(x)d\rho(x).
   \end{equation}

   By eventual strict monotonicity of $b$ we have
   \begin{equation}\begin{split}\label{eq:top_est}
     \int_{\R^d}\lambda^{+}(x)d\rho(x) & =\int_{B_{R}}\lambda^{+}(x)d\rho(x)+\int_{B_{R}^{c}}\lambda^{+}(x)d\rho(x)\\
         & \le\|Db\|_{C^{0}(B_{R})}\rho(B_{R})-\lambda_1\rho(B_{R}^{c}).
   \end{split}\end{equation}
   Next, we will prove that for $\s>>1$ the invariant measure $\rho$ ``flattens'', i.e.\ for each $\blue{\tilde R}\ge 0$, $\rho(B_{\blue{\tilde R}})\to0$ for $\s\to\infty$. Thus, the right hand side in \eqref{eq:top_est} becomes negative for $\s$ large enough, which finishes the proof.
   
\textit{Step 2}: For each $\blue{\tilde R}\ge 0$, $\rho(B_{\blue{\tilde R}})\to0$ for $\s\to\infty$. 

Indeed: Given $\s>0$ let $\rho^\s$ be the corresponding invariant measure, thus solving the Fokker-Planck equation
\[
\frac{\s^{2}}{2}\Delta\rho^{\s}-\mathrm{div}(b\rho^{\s})=0
\]
in distributional sense. Consequently, for all $\varphi\in C_{c}^{\infty}$ we have
\[
\int(\Delta\varphi+\frac{2}{\s^{2}}b\cdot\nabla\varphi)\ d\rho^{\s}=0.
\]
Since $\rho^{\s}(\R^{d})=1$, there is a weakly$^{*}$ convergent subsequence $\rho^{\s_{n}}\rightharpoonup^{*}\rho$ in the space of all signed measures of total variation on $\R^{d}$. Clearly, $\rho(\R^{d})\le1$. Since, 
\[
-\int(\frac{2}{\s^{2}}b\cdot\nabla\varphi)\ d\rho^{\s}\le\frac{2}{\s^{2}}\|b\cdot\nabla\varphi\|_{C^{0}}\rho^{\s}(\R^{d})=\frac{2}{\s^{2}}\|b\cdot\nabla\varphi\|_{C^{0}}
\]
we have 
\[
\int\Delta\varphi\ d\rho^{\s_{n}}\le\frac{2}{\s_{n}^{2}}\|b\cdot\nabla\varphi\|_{C^{0}}.
\]
Taking the limit yields $\int\Delta\varphi\ d\rho\le0$, for all $\varphi\in C_{c}^{\infty}$. Thus, also $\int\Delta\varphi\ d\rho=0$, for all $\varphi\in C_{c}^{\infty}$. We next show that this implies $\rho=0$. Let $\varphi_{\l}(x)=e^{-\l|x|^{2}}$ and note 
\begin{align*}
\Delta\varphi_{\l}(x) & =2\l\varphi_{\l}(x)(2\l|x|^{2}-d).
\end{align*}
A simple approximation/cut-off argument implies $\int\Delta\varphi_{\l}\ d\rho=0$. 
Given $\blue{\tilde R}>0$ we can choose $\l$ small enough such that $-\Delta\varphi_{\l}(x)\ge\l$ for all $x\in B_{\blue{\tilde R}}$. Let $\blue{\tilde R}(\l)=\sqrt{\frac{d}{2\l}}$ and note $\blue{\tilde R}(\l)\ge \blue{\tilde R}$ for all $\l$ small enough. Then
\begin{align*}
\rho(B_{\blue{\tilde R}}) & \le\frac{1}{\l}\int_{B_{\blue{\tilde R}}}-\Delta\varphi_{\l}\, d\rho\\
 & \le\frac{1}{\l}\int_{B_{\blue{\tilde R}(\l)}}-\Delta\varphi_{\l}\, d\rho\\
 & =\frac{1}{\l}\int_{B_{\blue{\tilde R}(\l)}^{c}}\Delta\varphi_{\l}\, d\rho.
\end{align*}
Note that $\frac{1}{\l}\Delta\varphi_{\l}(x)\le4e^{-\l|x|^{2}}\l|x|^{2}\le C$ for some constant $C$ independent of $\l\ge0$. Since $\blue{\tilde R}(\l)\to\infty$ for $\l\to0$ and $\rho(\R^{d})\le1$ we obtain
\begin{align*}
\rho(B_{\blue{\tilde R}}) & \le C\rho(B_{\blue{\tilde R}(\l)}^{c})\to0
\end{align*}
for $\l\to0$. Hence, $\rho=0$ and thus $\rho^{\s_{n}}\rightharpoonup^{*}0$, which finishes the proof. 
\end{proof}

We next consider the case of SDE with gradient structure, i.e.
\begin{equation}\label{eq:gradient_SDE-2}
   dX_t  =-\nabla V\left(  X_t  \right)  dt+\sigma
   dW_t\quad\text{on }\R^d,
\end{equation}
with $\sigma>0$, $V\in C_{loc}^{2,\delta}\left(  \mathbb{R}^{d}\right)$ for some $\d>0$ and $b:=-\nabla V$ satisfying \eqref{eq:b_mon}. By \cite{DS11} there is an associated white noise RDS $\varphi$ to \eqref{eq:gradient_SDE-2}. 

If $\varrho(x):=e^{-\frac{2}{\sigma^2}V(x)}\in L^1(\R^d)$, then by \cite[Theorem, p.243]{V80}\footnote{In fact, \cite[Theorem, p.243]{V80} assumes $b$ to be smooth. However, it is an easy exercise to see that only $b \in C^\d$ for some $\d>0$ is required for the proof (cf.\ also \cite[Theorem 10.4.1]{K96} for an according regularity result for linear, non-degenerate second order PDE with H\"older coefficients).}, the Markovian semigroup defined by $P_tf(x):=\E f(\varphi_t(\cdot,x))$ has 
  $$\rho = \frac{1}{Z_{\sigma}}e^{-\frac{2}{\sigma^2}V(x)}dx$$
as an invariant probability measure, where $Z_{\sigma}:=\int_{\R^d}e^{-\frac{2}{\sigma^2}V(x)}dx$. In this case, by \cite[Theorem 3]{S84}, $P_t$ is strongly mixing with ergodic measure $\rho$. Thus, by Lemma \ref{lem:stable_mfd} and \eqref{eq:der_bound} the top Lyapunov exponent $\l_{top}$ is well-defined and it only remains to show $\l_{top}<0$.

\begin{ex}\label{ex:grad_ex}
Consider the SDE in $\mathbb{R}^{d}$%
\[
dX_t  =-\nabla V\left(  X_t\right)  dt+\sigma
dW_t  ,
\]
with $\sigma>0$, $V\in C_{loc}^{2,\delta}\left(  \mathbb{R}^{d}\right)$ for some $\d>0$ and $b:=-\nabla V$ satisfying \eqref{eq:b_mon}. Further assume that
\begin{equation}\begin{split}\label{eq:V_asspt}
    V\left(  x\right)  &\geq C_{0}\log\left\vert x\right\vert\\
        \left\Vert D^{2}V\left(  x\right)  \right\Vert &\leq C_{0}\left\vert
           x\right\vert ^{N},
\end{split}\end{equation}
for all $\left\vert x\right\vert \geq R_{0}$ and some $C_0,R_0 >1,N\ge 0$ and that
  \[
     \inf\left\{  \min_{\left\vert r\right\vert =1}\left(  D^{2}V\left(  x\right)
     r,r\right)  :x\text{ global minimum of }V\right\}  >0.
     \]
 Then  $\lambda_{top}<0$ for $\sigma$ small enough.
\end{ex}

\begin{proof} 
By \eqref{eq:V_asspt}, $\varrho(x):=e^{-\frac{2}{\sigma^2}V(x)}\in L^1(\R^d)$ for $\sigma$ small enough. Hence, as we have seen in Section \ref{sec:weak_sync}, there is a corresponding white noise RDS $\varphi$ with strongly mixing invariant probability measure $\rho$. Recall that 
$$d\rho\left(  x\right)  =\frac{1}{Z_{\sigma}}e^{  -\frac
{2}{\sigma^{2}}V\left(  x\right)  }  dx,$$
where $Z_{\sigma}=\int%
e^{-\frac{2}{\sigma^{2}}V\left(  x\right) } dx$. This
integral is finite for $\sigma$ small enough, because $V\left(  x\right)  \geq
C_{0}\log\left\vert x\right\vert $ for large $x$. Let $\mathcal{M}$ denote the
set of global minima of $V$. Without loss of generality, we may assume $V=0$
on $\mathcal{M}$ (hence $V\geq0$ on $\mathbb{R}^{d}$) and $0\in\mathcal{M}%
$. We also have $DV=0$ on $\mathcal{M}$. 

\textit{Step 1:} We prove that, for some constant $C>0$, we have
\[
Z_{\sigma}\geq C\sigma^{d}\quad \forall \sigma\in (0,1].
\]

Let $C_{1}:=\sup_{B\left(  0,1\right)
}\left\Vert D^{2}V\right\Vert $. 
For $x\in B\left(  0,1\right)  $ we have%
\[
V\left(  x\right)  =\frac{1}{2}\left(  D^{2}V\left(  \theta_{x}x\right)
x,x\right)  \leq C_{1}\left\vert x\right\vert ^{2}%
\]
for some $\theta_{x}\in\left(  0,1\right)  $. Hence, for $x\in B\left(
0,\sigma\right)$, we have%
\begin{align*}
V\left(  x\right)    & \leq C_{1}\sigma^{2}\\
e^{ -\frac{2}{\sigma^{2}}V\left(  x\right) }    & \geq
e^{ -2C_{1}}
\end{align*}
and therefore
\[
Z_{\sigma}\geq\int_{B\left(  0,\sigma\right)  }e^{ -2C_{1}}
dx=C\sigma^{d}%
\]
for a suitable constant $C>0$.

\textit{Step 2:} We prove that, for every $R\geq R_{0}$,
\[
\lim_{\sigma\rightarrow0}\frac{1}{Z_{\sigma}}\int_{B\left(  0,R\right)  ^{c}%
}\left(  1+\left\Vert D^{2}V\left(  x\right)  \right\Vert \right)  e^{
-\frac{2}{\sigma^{2}}V\left(  x\right)  } dx=0.
\]
We have (using Step 1)%
\begin{align*}
 &\frac{1}{Z_{\sigma}}\int_{B\left(  0,R\right)  ^{c}}\left(  1+\left\Vert
 D^{2}V\left(  x\right)  \right\Vert \right) e^{  -\frac{2}{\sigma^{2}%
 }V\left(  x\right)  } dx\\
 &\leq\frac{1}{C\sigma^{d}}\int_{B\left(
 0,R\right)  ^{c}}\left(  1+C_{0}\left\vert x\right\vert ^{N}\right)
 \left\vert x\right\vert ^{-\frac{2C_{0}}{\sigma^{2}}}dx
\end{align*}
and the result follows by dominated convergence.

\textit{Step 3:} Let $\mathcal{U}$ be an open neighborhood of $\mathcal{M}$. We
prove that%
\[
\lim_{\sigma\rightarrow0}\frac{1}{Z_{\sigma}}\int_{\mathcal{U}}e^{
-\frac{2}{\sigma^{2}}V\left(  x\right)  } dx=1
\]
and that, for every $R\ge R_0$ such that $\mathcal{U\subset}B\left(  0,R\right)  $,%
\[
\lim_{\sigma\rightarrow0}\frac{1}{Z_{\sigma}}\int_{B\left(  0,R\right)
\backslash\mathcal{U}}e^{ -\frac{2}{\sigma^{2}}V\left(  x\right)
} dx=0.
\]

We have $k=\inf_{\mathcal{U}^{c}}V>0$. Let $\mathcal{U}^{\prime}%
\mathcal{\subset U}$ be such that $V\left(  x\right)  \leq\frac{k}{2}$ for all
$x\in\mathcal{U}^{\prime}$. Then%
\begin{align*}
\int_{\mathcal{U}}e^{  -\frac{2}{\sigma^{2}}V\left(  x\right)  }
dx  & \geq\left\vert \mathcal{U}^{\prime}\right\vert e^{  -\frac
{k}{\sigma^{2}}}  \\
\int_{B\left(  0,R\right)  \backslash\mathcal{U}}e^{ -\frac{2}%
{\sigma^{2}}V\left(  x\right) }  dx  & \leq\left\vert B\left(
0,R\right)  \right\vert e^{  -\frac{2k}{\sigma^{2}}}  .
\end{align*}
Hence%
\[
\int_{B\left(  0,R\right)  \backslash\mathcal{U}}e^{  -\frac{2}%
{\sigma^{2}}V\left(  x\right)  }  dx\leq g\left(  \sigma\right)
\int_{\mathcal{U}}e^{  -\frac{2}{\sigma^{2}}V\left(  x\right) }
dx
\]
where $g\left(  \sigma\right)  :=\frac{\left\vert B\left(  0,R\right)
\right\vert e^{  -\frac{2k}{\sigma^{2}}}  }{\left\vert
\mathcal{U}^{\prime}\right\vert e^{  -\frac{k}{\sigma^{2}}}
}\rightarrow0$ as $\sigma\rightarrow0$. \ Moreover, we have seen in Step 2
that
\[
\lim_{\sigma\rightarrow0}\frac{1}{Z_{\sigma}}\int_{B\left(  0,R\right)  ^{c}%
}e^{  -\frac{2}{\sigma^{2}}V\left(  x\right)  }  dx=0.
\]
Therefore%
\begin{align*}
&\frac{1}{Z_{\sigma}}\int_{\mathcal{U}}e^{  -\frac{2}{\sigma^{2}%
}V\left(  x\right)  }  dx   \\
& =1-\frac{1}{Z_{\sigma}}\int_{B\left(
0,R\right)  \backslash\mathcal{U}}e^{  -\frac{2}{\sigma^{2}}V\left(
x\right)  }  dx-\frac{1}{Z_{\sigma}}\int_{B\left(  0,R\right)  ^{c}}%
e^{  -\frac{2}{\sigma^{2}}V\left(  x\right)  }  dx\\
& \geq1-g\left(  \sigma\right)  \frac{1}{Z_{\sigma}}\int_{\mathcal{U}}%
e^{  -\frac{2}{\sigma^{2}}V\left(  x\right)  }  dx-\frac
{1}{Z_{\sigma}}\int_{B\left(  0,R\right)  ^{c}}e^{  -\frac{2}%
{\sigma^{2}}V\left(  x\right)  }  dx
\end{align*}
and the result follows by dominated convergence. The proof of the second claim is similar.

\textit{Step 4:} We may now complete the proof. Under our assumptions, there
exists an open neighborhood $\mathcal{U}$ of $\mathcal{M}$ such that
$c:=\inf\left\{  \min_{\left\vert r\right\vert =1}\left(  D^{2}V\left(
x\right)  r,r\right)  :x\in\mathcal{U}\right\}  >0$. Since $\lambda^{+}\left(
x\right)  =-\min_{\left\vert r\right\vert =1}\left(  D^{2}V\left(  x\right)
r,r\right)  $, we have%
\[
\lambda^{+}\left(  x\right)  \leq\left\Vert D^{2}V\left(  x\right)
\right\Vert \qquad\text{for all }x\in\mathbb{R}^{d}%
\]%
\[
\lambda^{+}\left(  x\right)  \leq-c\qquad\text{for all }x\in\mathcal{U}%
\text{.}%
\]
Hence, for $R\geq R_{0}$ such that $\mathcal{U\subset}B\left(  0,R\right)
$, we have%
\begin{align*}
\int\lambda^{+}\left(  x\right)  d\rho\left(  x\right)    & \leq\frac
{1}{Z_{\sigma}}\int_{B\left(  0,R\right)  ^{c}}\left\Vert D^{2}V\left(
x\right)  \right\Vert e^{  -\frac{2}{\sigma^{2}}V\left(  x\right)
}  dx\\
& +\left(  \sup_{B\left(  0,R\right)  }\left\Vert D^{2}V\left(  x\right)
\right\Vert \right)  \frac{1}{Z_{\sigma}}\int_{B\left(  0,R\right)
\backslash\mathcal{U}}e^{  -\frac{2}{\sigma^{2}}V\left(  x\right)
}  dx\\
& -\frac{c}{Z_{\sigma}}\int_{\mathcal{U}}e^{  -\frac{2}{\sigma^{2}%
}V\left(  x\right)  }  dx.
\end{align*}
Form the results of the previous steps we get%
\[
\int\lambda^{+}\left(  x\right)  d\rho\left(  x\right)  <0
\]
for $\sigma$ small enough, hence $\lambda_{top}<0$ by the same arguments as
used in the proof of Example \ref{ex:lyapunov_large_noise}.
\end{proof}

We next consider SDE of the type \eqref{eq:gradient_SDE-2} with radially symmetric potential. Note that we neither need to assume $\s$ small nor an assumption of the type \eqref{eq:V_asspt} here.

\begin{ex}\label{ex:grad_SDE_symm}
  Consider the SDE%
  \[
  dX_t  =-\nabla V\left(  X_t\right)  dt+\sigma
  dW_t\quad \text{on }\mathbb{R}^{d},
  \]
  with $\sigma>0$ and $b:=-\nabla V$ satisfying \eqref{eq:b_mon}. Further assume that $V$ is radially symmetric with $V(x)=g(|x|^2)$, $g\in C^{2,\d}_{loc}$ being a convex function and $\varrho(x)=e^{-\frac{2}{\sigma^2}V(x)}\in L^1(\R^d)$. Then  $\lambda_{top}<0$.
\end{ex}
\begin{proof}
\textit{Case $d=1$:} Let $x_0 := 0$ and $x_n := \inf\{x\ge x_{n-1}:e^{-\frac{2}{\sigma^2}V(x)} \le \frac{1}{n}\}$. Since $e^{-\frac{2}{\sigma^2}V(x)}\in L^1(\R^d)$ we have  $x_n < \infty$ for all $n\in \N$. Moreover, $x_n \to \infty$ and $V'(x_n) \ge 0$. By symmetry $V'(0)=0$. We conclude that
	 \begin{align*}
	              \int\lambda^{+}(x)d\rho(x) 
	              & =-\frac{1}{Z_{\sigma}}\int_\R V''(x)e^{-\frac{2}{\s^{2}}V(x)}dx\\
	              & =-\frac{2}{Z_\sigma}\lim_{n\to\infty}\int_0^{x_n} V''(x)e^{-\frac{2}{\s^{2}}V(x)}dx\\
	              &=-\frac{2}{Z_\sigma}\lim_{n\to\infty}\left(V'(x)e^{-\frac{2}{\s^{2}}V(x)}\big|_0^{x_n}+\frac{2}{\s^{2}}\int_0^{x_n} |V'(x)|^2e^{-\frac{2}{\s^{2}}V(x)}dx\right)\\
				  &\le -\frac{4}{Z_\sigma\s^{2}}\int_0^{\infty} |V'(x)|^2e^{-\frac{2}{\s^{2}}V(x)}dx\\
				  &<0,
	\end{align*}
	which implies $\lambda_{top}<0$ by the same arguments as used in the proof of Example \ref{ex:lyapunov_large_noise}.

It is, in fact, well-known that in case $d=1$ the claim is true under much weaker assumptions. Synchronization in this {\em monotone} case is 
discussed in \cite{FGS15}, for example.
	
	\textit{Case $d\ge 2$:} 
    Since $V(x)=g(|x|^{2})$ we compute
          \begin{align*}
          \nabla V(x) & =2g'(|x|^{2})x\\
          D^2V(x) & =4g''(|x|^{2})x\otimes x+2g'(|x|^{2})Id.
          \end{align*}
          Thus (using $g''\ge0$)
          \begin{align*}
          (D^{2}V(x)r,r) & = 4g''(|x|^{2})(x,r)^{2}+2g'(|x|^{2})|r|^{2}\\
          \min_{|r|=1}(D^{2}V(x)r,r) & =2g'(|x|^{2}).
          \end{align*}
          We note that
          \begin{align*}
             \int e^{-\frac{2}{\sigma^2}V(x)}dx 
             = \int e^{-\frac{2}{\sigma^2}g(|x|^2)}dx 
             =  C\int_{\R_+} r^{d-1} e^{-\frac{2}{\sigma^2}g(r^2)}dr
             <\infty.
          \end{align*}
          Hence, there is a sequence $t_n \uparrow\infty$ such that $t_n^{d-1} e^{-\frac{2}{\sigma^2}g(t_n^2)} \to 0$ for $n\to \infty$.
          We conclude, 
          \begin{align*}
             \int\lambda^{+}(x)d\rho(x) & =-\frac{2}{Z}\int g'(|x|^{2})e^{-\frac{2}{\s^{2}}g(|x|^{2})}dx\\
             & =-c\int_{\R_{+}}r^{d-1}g'(r^2)e^{-\frac{2}{\s^{2}}g(r^2)}dr\\
             & =c\int_{\R_{+}}r^{d-2}\frac{d}{dr}e^{-\frac{2}{\s^{2}}g(r^2)}dr\\
             & = c\lim_{n\to \infty} r^{d-2}e^{-\frac{2}{\s^{2}}g(r^2)}\big|_0^{t_n} - c(d-2)\int_{\R_{+}}r^{d-3}e^{-\frac{2}{\s^{2}}g(r^2)}dr\\
             &< 0,
          \end{align*}
    which implies $\lambda_{top}<0$ by the same arguments as used in the proof of Example \ref{ex:lyapunov_large_noise}.
\end{proof}

\subsubsection{Deducing asymptotic stability}\label{ssec:deduce_as}

As in the previous section we consider SDE with additive noise of the type \eqref{eq:lyapunov_add_SDE} with $\s>0$, $b \in C^{2}_{loc}(\R^d)$ and $b$ satisfying \eqref{eq:b_mon}. Assume that $\varphi$ is strongly mixing with invariant measure $\rho$ (cf.\ the previous section for appropriate conditions). In the last section we have introduced sufficient conditions for the corresponding RDS $\varphi$ to have negative top Lyapunov exponent. According to Corollary \ref{cor stability} this implies asymptotic stability for $\varphi$ if condition \eqref{eq:integrability_2} holds. Thus, we next present sufficient conditions implying \eqref{eq:integrability_2}.
\begin{lemma}\label{lem:as_stab}
Assume that $b\in C^2_{loc}(\R^d)$ \blue{satisfies \eqref{eq:b_mon} and}
\begin{equation}\label{eq:as_stab_cdt_1}
   \|D^2 b(x)\| \le C(|x|^M+1) \quad\forall x\in\R^d,
\end{equation}
for some $M\in\N,C\ge 0$. Further assume 
\begin{equation}\label{eq:as_stab_cdt_2}
\int_{\R^{d}}\log^{+}(|x|)d\rho(x)<\infty.
\end{equation}
Then
\begin{align*}
\E\int_{\R^{d}}\log^{+}(\|\varphi_{1}(\omega,\cdot+x)-\varphi_{1}(\omega,x)\|_{C^{1,\d}(\bar{B}(0,1))})d\rho(x) & <\infty,
\end{align*}
for every $\d \in (0,1)$.
\end{lemma}
\begin{proof}\blue{We first note that $\varphi_t(\omega,\cdot)\in C^2_{loc}$ since $b\in C^2_{loc}$. Indeed, this follows by considering the transformation $\tilde \varphi_t(\omega,x) := \varphi_t(\omega,x)-\sigma W_t(\omega)$ satisfying
  $$\frac{d}{dt}\tilde \varphi_t = b(\tilde \varphi_t+\sigma W_t).$$
Then, $\tilde \varphi_t(\omega,\cdot) \in C^2_{loc}$ follows by arguments similar to \cite[Theorem 2.10]{T12}.
}

Let 
\[
F^{x}(\omega,y):=\varphi_{1}(\omega,y+x)-\varphi_{1}(\omega,x).
\]
We aim to estimate $\|F^{x}(\omega,\cdot)\|_{C^{1,\d}(\bar{B}(0,1))}$. Due to \eqref{eq:der_bound} we have that
\[
|F^{x}(\omega,y)|=|\varphi_{1}(\omega,y+x)-\varphi_{1}(\omega,x)|\le e^{\l},\quad\forall y\in\bar{B}(0,1)
\]
and since $DF^{x}(\omega,y)(v):=D\varphi_{1}(\omega,y+x)(v)$ we obtain that $$\|DF^{x}(\omega,y)\|=\|D\varphi_{1}(\omega,y+x)\|\le e^{\l}.$$ It remains to estimate $\|DF^{x}(\omega,\cdot)\|_{C^{\d}(\bar{B}(0,1))}$. First note that \blue{(here and in the following $C$ denotes a generic constant that may change value from line to line)}
\begin{align*}
  \|DF^{x}(\omega,\cdot)\|_{C^{\d}(\bar{B}(0,1))}
  &\le \|DF^{x}(\omega,\cdot)\|_{C(\bar{B}(0,1))} + C\|D^2 F^{x}(\omega,\cdot)\|_{C(\bar{B}(0,1))}\\
  &\le e^\l+ C\|D^2 \varphi_1(\omega,\cdot)\|_{C(\bar{B}(x,1))}.
\end{align*}
In the following, for simplicity, we suppress the $\omega$-dependence in the notation. Since, \blue{for all $z\in\bar B(x,1)$ and} all $v,w\in\R^d$ with $|v|,|w| \le 1$,
  $$\frac{d}{dt}D^{2}\varphi_{t}\left(  \blue{z}\right)  \left(  v,w\right)    
  =D^{2}b\left(  \varphi_{t}\left(  \blue{z}\right)  \right)  \left(  D\varphi
  _{t}\left(  \blue{z}\right)  v,D\varphi_{t}\left(  \blue{z}\right)  w\right)  +Db\left(
  \varphi_{t}\left(  \blue{z}\right)  \right)  D^{2}\varphi_{t}\left(  \blue{z}\right)
  \left(  v,w\right),$$
we have that
\begin{align*}
\frac{1}{2}\frac{d}{dt}\left\vert D^{2}\varphi_{t}\left(  \blue{z}\right)  \left(
v,w\right)  \right\vert ^{2}  & \leq\left\Vert D^{2}b\left(  \varphi
_{t}\left(  \blue{z}\right)  \right)  \right\Vert \left\Vert D\varphi_{t}\left(
\blue{z}\right)  \right\Vert ^{2}\left\vert D^{2}\varphi_{t}\left(  \blue{z}\right)  \left(
v,w\right)  \right\vert \\
& +\left\langle Db\left(  \varphi_{t}\left(  \blue{z}\right)  \right)  D^{2}%
\varphi_{t}\left(  \blue{z}\right)  \left(  v,w\right)  ,D^{2}\varphi_{t}\left(
\blue{z}\right)  \left(  v,w\right)  \right\rangle \\
&\leq e^{2\lambda t}C\left(  \left\vert \varphi_{t}\left(  \blue{z}\right)
\right\vert ^{M}+1\right)  \left\vert D^{2}\varphi_{t}\left(  \blue{z}\right)
\left(  v,w\right)  \right\vert +\lambda\left\vert D^{2}\varphi_{t}\left(
\blue{z}\right)  \left(  v,w\right)  \right\vert ^{2} \\
&\leq e^{4\lambda t}C\left(  \left\vert \varphi_{t}\left(  \blue{z}\right)
\right\vert ^{M}+1\right)  ^{2}+\left(  \lambda+1\right)  \left\vert
D^{2}\varphi_{t}\left(  \blue{z}\right)  \left(  v,w\right)  \right\vert ^{2}.
\end{align*}%
Hence,
\begin{align*}
\left\vert D^{2}\varphi_{t}\left(  \blue{z}\right)  \left(  v,w\right)  \right\vert
^{2}  & \leq\int_{0}^{t}e^{2\left(  \lambda+1\right)  \left(  t-s\right)
}e^{4\lambda s}C\left(  \left\vert \varphi_{s}\left(  \blue{z}\right)  \right\vert
^{M}+1\right)  ^{2}ds\\
& \leq C\int_{0}^{t}\left\vert \varphi_{s}\left(  \blue{z}\right)  \right\vert
^{2M}ds+C,%
\end{align*}
for all $t\in[0,1]$ and thus
  \[
  \| D^2\varphi_{1}(z)\|^2 \leq C\int_{0}^{1}| \varphi_{s}(z)|^{2M}ds+C.\]
Since 
\[
\left\vert \varphi_{s}\left(  \omega,z\right)  \right\vert \leq\left\vert
\varphi_{s}\left(  \omega,x\right)  \right\vert +e^{\lambda},
\]
for all $z\in\bar B(x,1)$ and $s\in [0,1]$, we have
\[
  \| D^2\varphi_{1}(\omega)\|^2_{C(\bar B(x,1))} \leq C\int_{0}^{1}| \varphi_{s}(\omega,x)|^{2M}ds+C.\]
In conclusion,
$$\|F^{x}(\omega,\cdot)\|_{C^{1,\d}(\bar{B}(0,1))} 
\le  C\int_{0}^{1}| \varphi_{s}(\omega,x)|^{2M}ds+C.
$$
Hence, using Fubini's theorem and Jensen's inequality \blue{(taking $C\ge1$)},
\begin{align*}
& \E\int_{\mathbb{R}^{d}}\log^{+}\left( \|F^{x}(\omega,\cdot)\|_{C^{1,\d}(\bar{B}(0,1))} \right)  d\rho(x)  \\
& \leq \E\int_{\mathbb{R}^{d}}\log^{+}\left(  \gray{C\left\vert \varphi_{1}\left(
x\right)  \right\vert ^{2}+C+}\blue{C\int_{0}^{1}\left\vert \varphi_{s}\left(\omega,
x\right)  \right\vert ^{2M}ds+C}\right)  d\rho\left(  x\right)  \\
& \leq\int_{\mathbb{R}^{d}}\log^{+}\left(  \gray{C\E\left[  \left\vert \varphi
_{1}\left(  x\right)  \right\vert ^{2}\right]  +C+}\blue{C\int_{0}^{1}\E\left[
\left\vert \varphi_{s}\left(\omega,  x\right)  \right\vert ^{2M}\right]  ds+C}\right)
d\rho\left(  x\right).
\end{align*}
We note that \eqref{eq:b_mon} implies that there is a constant $C>0$ such that
  $$(b(x),x)\le C(|x|^2+1)\quad \forall x\in \R^d. $$
Hence, an application of It\^o's formula yields 
\[
\sup_{r\in[0,1]}\E|\varphi_{r}(\omega,x)|^{2p}\le C(|x|^{2p}+1),
\]
for all $p\ge 1$ (where $C$ depends on $p$). Thus, using \eqref{eq:as_stab_cdt_2} we conclude that
\begin{align*}
 \E\int_{\mathbb{R}^{d}}\log^{+}\left(  \|F^{x}(\omega,\cdot)\|_{C^{1,\d}(\bar{B}(0,1))}^{\gray{2}}\right)  \rho\left(
dx\right)  
&\leq\int_{\mathbb{R}^{d}}\log^{+}\left( \gray{ C\left\vert x\right\vert
^{2}+C+}\blue{C\left\vert x\right\vert ^{2M}+C}\right)  d\rho\left(  x\right) \\
&<\infty,
\end{align*}
which finishes the proof.
\end{proof}

\subsection{Properties of dissipativity, contraction and swift transitivity}\label{sec:eventual monotonicity}

This section is devoted to the proof of contraction on large sets and swift transitivity. We consider SDE with additive noise
\begin{equation}\label{eq:add_swift}
   dX_t = b(X_t)dt+\s dW_t,
\end{equation}
where $b$ is locally Lipschitz and satisfies \eqref{eq:b_mon}, $\s>0$.

\begin{pro}\label{prop:swift-SDE}
Let $\varphi$ be the RDS associated to \eqref{eq:add_swift}. \blue{Then, for all balls $B(x,r)$ and all $\delta >0$, $z\in \R^d$, there is a time $T_0>0$ such that for all 
$t_0\in (0,T_0]$ one has 
\[
\PP\big(\varphi_{t_0}\left(\cdot, x'  \right)  \in B\left(x'+z,\delta\right),\ \forall x' \in B(x,r)\big) >0.
\]}
In particular, the swift transitivity property holds. 

Assume, in addition, that $b$  is monotone on large sets, i.e.\ for each $r>0$ there exists some $z \in{\mathbb{R}}^{d}$ such that 
\[
  ( b\left(  x\right)  -b\left(  y\right)  ,x-y)
  <0\qquad\text{for all } x \neq y, x,y \in B(z,r).
\]
Then the property of contraction on large sets holds.
\end{pro}

\begin{proof}

\textbf{Part 1} (swift transitivity). 
\blue{Fix $x,z \in \R^d$ and $r,\delta>0$.% and define $y:=x+z$.  

Let $B:=\sup_{|v|\le r+|z|+1}|b(x+v)|$, and $T_0:=\frac{\delta\wedge 1}{4B} \wedge 1$. 

Fix $t_0 \in (0,T_0]$ and define}
\[
\psi (t)  :=x+\frac{t}{t_0} \blue{z} ,\qquad t\in [0,t_0],
\]
and 
\[
f(t)  := \frac{1}{\s}\left(\psi (t)
-x-\int_{0}^{t}b\left(  \psi (s)  \right)  ds\right),\qquad
t\in [0,t_0].
\]

\blue{Abusing notation we write $\varphi_t(g,y),\,t \in [0,t_0]$ for the solution of \eqref{eq:add_swift} with initial condition $y$ driven by $g$ instead of $W$ where 
$g \in C_0:=\{h \in  C([0,t_0],\R^d):g(0)=0\}$.

Then $\varphi_{t}\left( f ,x\right)  =\psi \left(  t\right)  $, $t\in\left[  0,t_0\right]  $ and, in
particular, $\varphi_{t_0}\left( f,x\right)  =x+z$.

The map $g\mapsto\varphi_{.}\left(g,y\right)  $ is continuous
from $C_0$ to  $C([0,t_0],\R^d)$   with respect to the supremum norm $\|\cdot\|$. Therefore, there exists some 
$\varepsilon \in (0,(\delta \wedge  1)/(2\sigma))$ such that for $g \in C_\varepsilon :=\{\bar g \in C_0:\|\bar g - f\|\le \varepsilon\}$ 
we have 
$$
\| \varphi_{.}\left( g,x\right)-\varphi_{.}\left(  f,x\right)  \| \leq \delta/2.
$$
Let $x' \in B(x,r)$ and
$$\tau (g):=\inf\{t \ge 0: |\varphi_{t}(g, x')-x'|=|z|+1\}$$
and assume that there exists some $g \in C_\varepsilon$ for which $\tau:=\tau(g)<t_0$. Then
\begin{align*}
|z|+1&=|\varphi_{\tau}(g, x')-x'|\le Bt_0 +\sigma |g(\tau)|\le Bt_0 + \sigma(|f(\tau)|+\varepsilon)\\
&\le Bt_0 + \sigma \left( \sup_{0\le s \le t_0}|f(s)|+\varepsilon\right)\\
&\le Bt_0+ |z|+ Bt_0 + \sigma \varepsilon < 
\frac 14+|z| + \frac 14 + \frac 12 =|z|+1,
\end{align*}
which is a contradiction. Hence, for $g \in C_\varepsilon$, we have
\begin{align*}
|\varphi_{t_0}(g,x')-(x'+z)|&\le |\varphi_{t_0}(g,x')-x'-\sigma f(t_0)| + |\sigma f(t_0) -z|\\
&\le Bt_0+ \sigma |g(t_0)-f(t_0)| + |\sigma f(t_0)-z|\\
&\le Bt_0+\sigma \varepsilon + Bt_0 \le \delta.
\end{align*}
Since $\PP(W|_{[0,t_0]}\in C_\varepsilon)>0$ the first claim in the proposition follows}.

\textbf{Part 2} (contraction on large sets). 

Let $R>0$. By assumption there is a $z\in\R^d$ such that 
\begin{equation}\label{eq:strong_mon_proof}
  ( b\left(  x\right)  -b\left(  y\right)  ,x-y)
    <0\qquad\text{for all } x \neq y, x,y \in B(z,3R).
\end{equation}
Let
\[
\omega^0\left(  t\right)  := -\frac{tb(z)}{\s}\quad t\ge 0.
\]
Then $\varphi_{t}\left(  \omega^0,z\right)  = z$ for all $t\ge 0$. Due to \eqref{eq:strong_mon_proof}, 
  $$t\mapsto |\varphi_t(\omega^0,x)-\varphi_t(\omega^0,z)|=|\varphi_t(\omega^0,x)-z| $$
   is non-increasing for all $x \in B(z,R)$. Hence,
   $$\varphi_t(\omega^0,B(z,R))\subseteq B(z,R)\quad\forall t\ge 0.$$ 
By \eqref{eq:strong_mon_proof}, there is a $c>0$ such that
\begin{equation}\label{eq:gen_mon}
   ( b\left(  x\right)  -b\left(  y\right)  ,x-y)
     <-c|x-y|^2\qquad\forall x \neq y, x,y \in B(z,2R), |x-y|\ge \frac{R}{9}.
\end{equation}
Let $T_0>0$ such that $e^{-cT_0}\le \frac{1}{9}$. Further, let $x \in B(z,R)$ and define
  $$\tau := \inf\{t\ge 0: |\varphi_t(\omega^0,x)-z|\le \frac{R}{9}\}.$$
Due to \eqref{eq:gen_mon} we have that
 $$d|\varphi_t(\omega^0,x)-\varphi_t(\omega^0,z)|^2 \le -2c |\varphi_t(\omega^0,x)-\varphi_t(\omega^0,z)|^2dt \quad \text{on } [0,\tau).$$
By Gronwall's Lemma this implies
  $$|\varphi_{t}(\omega^0,x)-z|=|\varphi_{t}(\omega^0,x)-\varphi_{t}(\omega^0,z)| \le  e^{-ct} R, \quad\forall t\in[0,\tau)$$
and thus $\tau \le T_0$. Since  $t\mapsto |\varphi_t(\omega^0,x)-z| $ is non-increasing, we conclude that
   $$|\varphi_{T_0}(\omega^0,x)-z|\le \frac{R}{9}.$$
By continuity of $\omega \mapsto \varphi_{T_0}(\omega,\cdot)\in C(\bar{B}(z,R);\R^d)$ this implies
 $$\PP\left(\varphi_{T_0}(\cdot,B(z,R))\subseteq B\left(z,\frac{R}{8}\right) \right)>0$$
and thus 
 $$\PP\left(\diam\left(\varphi_{T_0}(\cdot,B(z,R))\right)\le\frac{R}{4} \right)>0.$$
\end{proof}

\blue{
\begin{rk}
Note that we did not use the one-sided Lipschitz property \eqref{eq:b_mon} in the proof of the first part of Proposition \ref{prop:swift-SDE}, so that the statement even holds assuming only that the drift $b$ is locally Lipschitz continuous.
\end{rk}}

\subsection{Weak synchronization for gradient-type SDE}\label{sec:weak_synchr_SDE}

In the case of gradient-type SDE we can use the results from Section \ref{sec:weak_sync} in order to deduce weak synchronization without assuming contraction of large balls 
(as compared to Section \ref{sec:synchronization}). In fact, we will prove that \eqref{eq:ptw_stab} is always satisfied as soon as there is an invariant measure.

More precisely, consider the SDE 
\begin{equation}\label{eq:gradient_SDE}
  dX_t=-\nabla V(X_t)dt+\sigma dW_t\quad\text{on } {\mathbb{R}}^d,
\end{equation}
with $V \in C^2({\mathbb{R}}^d,{\mathbb{R}})$, $\sigma >0$ and $b:=-\nabla V$ satisfying assumption \eqref{eq:b_mon}. Further assume $\varrho(x):=e^{-\frac{2}{\sigma^2}V(x)}\in L^1(\R^d)$. As seen in Section \ref{sec:add_noise}, there is an associated white noise RDS $\varphi$ to 
\eqref{eq:gradient_SDE} and  $P_tf(x):=\E f(\varphi_t(\cdot,x))$ is strongly mixing with ergodic measure 
  $$\rho = \frac{1}{Z_\sigma}e^{-\frac{2}{\sigma^2}V(x)}dx,$$
where $Z_\sigma:=\int_{\R^d}e^{-\frac{2}{\sigma^2}V(x)}dx$.

\begin{thm}\label{thm:weak_synchronization_gradient}
  Assume that $\varrho(x):=e^{-\frac{2}{\sigma^2}V(x)}\in L^1(\R^d)$ and that $\varphi$ is weakly asymptotically stable on $U$ with $\rho(U)>0$. Then, there is a minimal 
  weak point attractor $A$ consisting of a single random point $a(\omega)$ and
    $$A(\omega)=\supp(\mu_\omega)=\{a(\omega)\}\quad  \PP\text{-a.s..}$$
\end{thm}

The proof of Theorem \ref{thm:weak_synchronization_gradient} \gray{will be}\blue{is} a simple consequence of the following Lemma and Theorem \ref{thm:weak_synchronization}. Note that by Proposition \ref{prop:swift-SDE} $\varphi$ is \blue{pointwise} strongly \gray{(pointwise) }swift transitive.

\begin{lemma}\label{lemma:gradient_stopping}
\begin{enumerate}
\blue{\item     
Assume that $\varrho(x):=e^{-\frac{2}{\sigma^2}V(x)}\in L^1(\R^d)$. Then, for each $v \in \R^d\backslash \{0\}$, there exists some 
$z \in {\mathbb{R}}^d$ such that 
\begin{equation}\label{Veeee}
  ( b(z)-b(z -v),v) <0.
\end{equation}

\item Let $b$ be locally Lipschitz continuous and satisfy \eqref{eq:b_mon}. Further, assume that for each $v \in \R^d\backslash \{0\}$, there exists some $z \in {\mathbb{R}}^d$ such that $b$ satisfies \eqref{Veeee} and that the SDE  \eqref{eq:add_swift} admits an invariant probability measure $\rho$. 
Then, for each pair $x,y \in {\mathbb{R}}^d$,  we have 
	$$\liminf_{t \to \infty} |\varphi_t(x)-\varphi_t(y)|=0,$$ 
almost surely.} 
\end{enumerate}
\end{lemma}

\begin{proof} (1):
\textit{Step 1}: We claim that for each \blue{$v \in \R^d\backslash \{0\}$} there exists some $z \in {\mathbb{R}}^d$, such that 
\begin{equation}\label{Veee}
V(z) < \frac 12 \big( V(z+\blue{v}) + V(z-\blue{v}) \big).
\end{equation}
Assume this is wrong for some particular  \blue{$v \in \R^d\backslash \{0\}$}, then for every $z \in {\mathbb{R}}^d$ we have
$$
V(z) \ge \frac 12 \big( V(z+\blue{v}) + V(z-\blue{v}) \big).
$$
Therefore, \blue{for each $z \in \R^d$}, one of the functions $n\mapsto V(z+n\blue{v}), V(z-n\blue{v})$ is non-increasing for $n\in \N_0$. 
\blue{Fix $z \in \R^d$.} Without loss of generality let $n\mapsto V(z+n\blue{v})$ be non-increasing. Let $n\in \N_0$. Due to the one-sided Lipschitz condition 
on $b$ the function $g(h) := V(z+n\blue{v}+ h\blue{v})+\frac{\l}{2}(h{|v|})^2$ is convex on $h \in [0,1]$. Moreover, $g(0) \le V(z)$ and 
$g(1) \le V(z)+\frac{\l}{2}\blue{|v|^2}$. Since $g$ is convex this implies $\sup_{h\in [0,1]}g(h)\le V(z)+\frac{\l}{2}\blue{|v|^2}$. 
Therefore, $\sup_{h\in [0,1]} V(z+n\blue{v}+h \blue{v})\le V(z)+\frac{\l}{2}\blue{|v|^2}$ for all $n\in \N_0$. Hence,  
$$
\gamma \mapsto V(z+\gamma \blue{v})
$$
where $\gamma \in [0,\infty)$, is bounded from above. In particular, $\int_{{\mathbb{R}}} \varrho(z+hv)\,\dd h=\infty$ holds for each $z \in {\mathbb{R}}^d$, and therefore 
$\varrho$ cannot be integrable.\\

\textit{Step 2}: 
\blue{To see the claim in the first part of the lemma, let $f(h):=V(z+hv)-V(z)$ with $z$ chosen such that \eqref{Veee} holds}. Due to \eqref{Veee} we have
\begin{equation}\label{eq:f_ineq}
  f(0)=0<\frac{1}{2}(f(\blue{1})+f(-\blue{1})).
\end{equation}
Assume that  $f'(\a) \le f'(\a-1)$ for all $\a \in [0,1]$. 
Integrating over $[0,\blue{1}]$ yields $f(\blue{1}) \le -f(-\blue{1})$ \blue{contradicting}  \eqref{eq:f_ineq}. \blue{Therefore, there exists some   
$\a \in [0,1]$ such that $f'(\a) > f'(\a-1)$. The first claim in the lemma follows after replacing $z+\alpha v$ by $z$.}

(2): \blue{The proof is inspired by the {\em controllability} approach which was used, for example, in \cite{AK87}.

For $x,y \in \R^d$, let $C(x,y)$ be the closure of the set of all $(u,v)\in \R^d \times \R^d$ for which there exists some $T \ge 0$ and $f \in C([0,T],\R^d)$ such that 
$f(0)=0$ and the (unique) solution of the pair 
\begin{align*}
  X_t=x+\int_0^tb(X_s)\,ds +\s f_t,\quad  Y_t=y+\int_0^tb(Y_s)\,ds +\s f_t
\end{align*}
satisfies $(X_T,Y_T)=(u,v)$. 

\textit{Step 1}: Fix $x\neq y \in \R^d$ and let $C:=C(x,y)$. We show that $C \cap \Delta \neq \emptyset$, where $\Delta$ is the diagonal in $\R^d \times \R^d$ as before. First, note that the first part of Proposition \ref{prop:swift-SDE} shows that $(u,v)\in C$ implies $(u+z,v+z) \in C$ for every $z \in \R^d$. Therefore the infimum in
$\delta:=\inf\{|u-v|: (u,v)\in C\}$ is actually attained, say at $(u_0,v_0)$. Assume that $\delta >0$ and let $z$ satisfy the assumption in the lemma for $v:=v_0-u_0$. 
Then $(z-v,z) \in C$ and \eqref{Veeee} holds which shows (using the control $f=0$)  that there exists $(\tilde u,\tilde v) \in C$ for which $|\tilde v - \tilde u|<\delta$ 
contradicting the definition of $\delta$. Therefore, $\delta =0$ and $C \cap \Delta \neq \emptyset$.\\

\textit{Step 2}: Fix $\varepsilon>0$. We show that for each $x,y \in \R^d$ we have that
$$
\liminf_{t \to \infty} |\varphi_t(x)-\varphi_t(y)|\le \varepsilon, \mbox{ almost surely},
$$
which obviously implies the statement in the second part of the lemma. 

For $x,y \in \R^d$ let
$$
\kappa(x,y):=\PP\big(\inf_{t\ge 0}|\varphi_t(\cdot,x)-\varphi_t(\cdot,y)|< \varepsilon\big),
$$
which is strictly positive by Step 1 and the fact that the function which maps the control $f$ to the solution is continuous in the supremum norm on $C([0,T],\R^d)$ for 
each $T>0$. Further, the map $(x,y) \mapsto \kappa(x,y)$ is lower semicontinuous since $\varphi$ is continuous. Therefore, $\kappa$ is bounded away from $0$ uniformly 
on compact subsets of $\R^d \times \R^d$. Now we use the fact that the RDS admits an invariant probability measure $\rho$. Let $K \subset \R^d$ be a compact set satisfying $\rho(K)>\frac{1}{2}$. By \cite[Lemma 2, Theorem 3]{S84} we know that $\rho$ and all transition probabilities for positive times are absolutely continuous with 
respect to the Lebesgue measure on $\R^d$ and therefore Birkhoff's ergodic theorem shows that for {\em every} pair $x,y\in \R^d$ we have that
\begin{align*}
  &\liminf_{T\to\infty}\frac{1}{T}\int_0^T 1_{K\times K}(\varphi_t(x),\varphi_t(y))dt\\
  &\ge \lim_{T\to\infty}\frac{1}{T} \int_0^T 1_K(\varphi_t(x))dt-\lim_{T\to\infty}\frac{1}{T} \int_0^T 1_{K^c}(\varphi_t(y))dt \\
  &= \rho(K)-\rho(K^c)>0,\quad\PP\text{-a.s..}
\end{align*}
Hence, for all $x,y\in \R^d$ and $n \in \N$ there exists some (random) 
$t \ge n$ such that $(\varphi_t(x),\varphi_t(y)) \in K\times K$ almost surely.  Define 
$$
\kappa:=\inf\{\kappa(x,y):x,y \in K\},
$$
which is strictly positive and 
$$
T(x,y):=\inf\left\{t \ge 0:  \PP\big(\inf_{s \in [0,t]}|\varphi_s(\cdot,x)-\varphi_s(\cdot,y)|< \varepsilon\big)  \ge \frac{\kappa}{2}\right\}.
$$
For fixed $x,y\in \R^d,\, x \neq y$, we define the  stopping times $S_n,T_n$ and the random variables $X_n,Y_n$, $n \in \N_0$ as follows:
\begin{align*}
S_0:&=0,\;X_0:=x,\;Y_0:=y\\
T_n:&=\inf\{t\ge S_n +1: (\varphi_t(x),\varphi_t(y))\in K \times K\}\\
X_n:&=\varphi_{T_n}(x),\;Y_n:=\varphi_{T_n}(y)\\
S_{n+1}:&=\inf\{t \ge T_n:|\varphi_t(x)-\varphi_t(y)|<\varepsilon\}  \wedge T(X_n,Y_n).
\end{align*}
Note that by the strong Markov property we have 
$$
\PP \big( |\varphi_{S_{n+1}}(x)-\varphi_{S_{n+1}}(y)|\le \varepsilon \, |\F_{T_n}  \big)    \ge \kappa/2,
$$
for every $n \in \N_0$ almost surely, so the assertion follows and the proof is complete.}
\end{proof}

\subsection{Summary, explicit examples and open problems}\label{sec:summary}

We close the paper by giving a short summary of general conditions on SDE for synchronization by noise and providing some explicit examples of our general results.

\begin{thm}[Synchronization for general drift]\label{thm:synchr_general_drift}
For SDE of the form \eqref{eq:SDE_intro}, with drift $b\in C^{2}_{loc}\left(  \mathbb{R}^{d}\right)$  satisfying \eqref{eq:b_mon}, \eqref{eq:as_stab_cdt_1} and assuming 
eventual strict monotonicity (cf.\ Definition \ref{def:eventual strict monotonicity}), we have synchronization by noise for sufficiently large noise intensity $\sigma$.\end{thm}
\begin{proof}
First, notice that eventual strict monotonicity implies monotonicity on large
sets (cf.\ Proposition \ref{prop:swift-SDE}) as well as \eqref{eqn:assumptions for existence of flow} as observed in Section \ref{sec:add_noise}. By \cite{DS11} this implies the 
existence of a (weak) random attractor. Condition \eqref{eq:b_mon} and the local Lipschitz property, implied by $b\in C_{loc}^{2}\left(  \mathbb{R}^{d}\right)$, yield the existence 
of an RDS $\varphi$. By \eqref{eqn:assumptions for existence of flow} and \cite[Theorem 4.3]{K12}, $\varphi$ is strongly mixing. Assumption \eqref{eq:as_stab_cdt_2} is a 
consequence of \eqref{eqn:assumptions for existence of flow} since, by It\^o's formula applied to $e^{\g|X_t|^2}$, one can even show $\int_{\R^d}e^{\g|x|^2}d\rho(x)<\infty$ for $\g$ 
small enough. Assumption \eqref{eq:as_stab_cdt_1}, \eqref{eq:as_stab_cdt_2}, $b\in C_{loc}^{2}\left(\mathbb{R}^{d}\right) $ and \eqref{eqn:assumptions for existence of flow} imply, 
via Lemma \ref{lem:as_stab}, the assumptions of Lemma \ref{lem:stable_mfd} and thus the existence of the Lyapunov spectrum, in particular $\lambda_{top}$. 
By Example \ref{ex:lyapunov_large_noise} and the property of eventual strict monotonicity we deduce $\lambda_{top}<0$ for large noise intensity $\sigma$. 
By Corollary \ref{cor stability} this implies asymptotic stability.

The local Lipschitz property and the monotonicity on large sets guarantee swift transitivity and contraction on large sets, by Proposition \ref{prop:swift-SDE}. Then, by
Theorem \ref{maintheorem} we deduce synchronization.
\end{proof}

For gradient systems, $b=-\nabla V$, condition \eqref{eq:b_mon} on $b$ can be expressed
more naturally as
\begin{equation}
\left(  D^{2}V\left(  x\right)  \xi,\xi\right) \ge - \lambda\left\vert
\xi\right\vert ^{2}\label{new on V}%
\end{equation}
for all $x,\xi\in\mathbb{R}^{d}$ and some $\lambda\geq0$.

\begin{thm}[Synchronization for gradient systems]
Consider an SDE of the form \eqref{eq:grad_SDE_intro} with
gradient drift $b=-\nabla V$. Then:
\begin{enumerate}
  \item  If $V\in C^{2}\left(  \mathbb{R}^{d}\right)$ satisfies
  condition (\ref{new on V}) and $e^{-\frac{2}{\sigma^{2}}V}\in L^{1}\left(
  \mathbb{R}^{d}\right)$, then \blue{ weak asymptotic stability on $U$ with $\rho(U)>0$ implies} weak synchronization\gray{ holds}.
  \item Let $V\in C_{loc}^{2,\delta}\left(\mathbb{R}^{d}\right)
  $ be such that there exists an $R>0$ for which\footnote{Note that
    this property is implied by \eqref{second new on V}.}
  \begin{equation}
  \left(  D^{2}V\left(  x\right)  \xi,\xi\right)  > 0\label{third new on V}%
  \end{equation}
  for all $\left\vert x\right\vert >R$ and $\xi\in\mathbb{R}^{d}$. Assume that there is a weak random attractor\footnote{The existence of (weak) random attractors under a weak coercivity condition has been obtained in \cite{DS11}.}. Then, \blue{asymptotic stability on some non-empty open set $U$} implies synchronization\gray{ holds}.
\end{enumerate}
\end{thm}

\begin{proof}
(1): Follows by Theorem \ref{thm:weak_synchronization_gradient}. (2): Condition (\ref{third new on V}) plus local boundedness of $\left\Vert D^{2}V\left(
x\right)  \right\Vert $ imply condition \eqref{new on V}. Condition
(\ref{new on V}), which implies \eqref{eq:b_mon} for $b=-\nabla V$, and the local
Lipschitz property of $b$, implied by  $V\in C_{loc}^{2}\left(  \mathbb{R}%
^{d}\right)$, yield the existence of an RDS. Then Proposition
\ref{prop:swift-SDE} applies because $b=-\nabla V$ is locally Lipschitz and
(\ref{third new on V}) implies monotonicity on large sets. By Theorem \ref{maintheorem} we deduce synchronization.
\end{proof}

\begin{thm}[Sufficient conditions for asymptotic stability of gradient systems]
Consider an SDE of the form \eqref{eq:grad_SDE_intro} with gradient drift $b=-\nabla V$. Assume that $V\in C_{loc}^{3}\left(  \mathbb{R}^{d}\right)$ satisfies conditions \eqref{eq:as_stab_cdt_1} and \eqref{new on V}.
\begin{enumerate}
  \item If $V$ satisfies condition \eqref{eq:V_asspt} and the positivity assumption on the infimum in Example \ref{ex:grad_ex}, then asymptotic stability holds for small noise intensity $\sigma>0$.
  
  \item If there exists an $R>0$ such that
  \begin{equation}
  \left(  D^{2}V\left(  x\right)  \xi,\xi\right)  \ge \lambda_{1}\left\vert
  \xi\right\vert ^{2}\label{second new on V}%
  \end{equation}
  for all $\left\vert x\right\vert >R$, $\xi\in\mathbb{R}^{d}$ and some $\lambda_{1}>0$, then asymptotic stability holds for large noise intensity $\sigma>0$.
  
  \item If $V\left(  x\right)=g\left(  \left\vert x\right\vert ^{2}\right)  $, $g\in C_{loc}^{3}$ is convex and \gray{we have }$\log^+(|x|) e^{-\frac{2}{\sigma^{2}}V(x)}\in L^{1}\left(\mathbb{R}^{d}\right)$\gray{ for some $\sigma>0$}, then asymptotic stability holds.
\end{enumerate}
\end{thm}
\begin{proof}
Conditions (\ref{new on V}) and $V\in C_{loc}^{3}\left(  \mathbb{R}^{d},\mathbb{R}\right)  $\ give us the existence of an RDS, as for (2) of the previous theorem. Assumption \eqref{eq:as_stab_cdt_2} holds in case (1) by \eqref{eq:V_asspt} (for small $\s$); it holds in case (2) since \eqref{second new on V} implies eventual strict monotonicity and thus \eqref{eqn:assumptions for existence of flow} and we have already remarked in the proof of Theorem \ref{thm:synchr_general_drift} that this implies \eqref{eq:as_stab_cdt_2}; and finally it holds in case (3) since we assume $\log^+(|x|) e^{-\frac{2}{\sigma^{2}}V(x)}\in L^{1}\left(\mathbb{R}^{d}\right)$. The assumptions of Lemma \ref{lem:as_stab} hold and thus Lemma \ref{lem:stable_mfd} and Corollary  \ref{cor stability} apply. In particular, $\lambda_{top}$ exists. In case (3), $V\in C_{loc} ^{3}\left(  \mathbb{R}^{d},\mathbb{R}\right)$ follows from $g\in C_{loc}^{3}$ and see the proof of Example \ref{ex:grad_SDE_symm} to realize that Lemma \ref{lem:stable_mfd} applies without assumption \eqref{eq:V_asspt}. Asymptotic stability follows by Corollary \ref{cor stability} as soon as we prove $\lambda_{top}<0$. Let us recall the proof in three cases. For (1) see Example \ref{ex:grad_ex} and for (3) see Example \ref{ex:grad_SDE_symm}. Finally, for (2), condition \eqref{second new on V} implies eventual strict monotonicity of $b$, hence $\lambda_{top}<0$ for large $\sigma$ by Example \ref{ex:lyapunov_large_noise}. 
\end{proof}

As in \cite{T08} our results apply to
\begin{align*}
 & V_{E}(x):=(0.5\text{\textminus}10e^{(\text{\textminus}|x\text{\textminus}p_{1}|^{2})}\text{\textminus}10e^{(\text{\textminus}|x\text{\textminus}p_{2}|^{2})})|x|^{2},\\
 & V_{S}(x):=(2\text{\textminus}5e^{(\text{\textminus}|x\text{\textminus}p_{3}|^{2})}\text{\textminus}6e^{(\text{\textminus}|x\text{\textminus}p_{4}|^{2})}\text{\textminus}7e^{(\text{\textminus}|x\text{\textminus}p_{5}|^{2})})|x|^{2},
\end{align*}
where $p_{1}=(0,1),p_{2}=(0,-1),p_{3}=(0,2),p_{4}=(2,-2),p_{5}=(-2,-2)$. In contrast to \cite{T08}, where only small noise $\s$ can be treated, our results also yield synchronization for large noise $\s$.

As pointed out in the introduction, the model example of a double-well potential
\begin{align*}
  V_D(x)=\frac{1}{4}|x|^4-\frac{1}{2}|x|^2,
\end{align*}
is not covered by the techniques in \cite{T08} for $d\ge 2$. In contrast, our results imply synchronization in this case for all $\s>0$. In particular, no restriction to small or large noise $\s$ is required here.

We close the paper by pointing out some open problems: In Theorem \ref{thm:weak_synchronization_gradient} we assumed that weak asymptotic stability holds. We leave it as an open problem whether this condition is always satisfied for gradient type SDE with additive noise \eqref{eq:grad_SDE_intro}.
Our general results may also be applied to infinite dimensional examples. In particular, synchronization for SPDE could be investigated. This will be subject of subsequent work. 
Numerical evidence suggests that the top Lyapunov exponent for the Lorentz system perturbed by strong noise (i.e.\ for $\s$ large) is negative and (weak) synchronization occurs. The Lorentz system, however, is not covered by the techniques put forward in Section \ref{sec:synchr_SDE}. We leave this as an open problem.
We prove swift transitivity for a large class of SDE with (non-degenerate) additive noise in Section \ref{sec:synchr_SDE}. It is left as an open problem to establish swift transitivity in other situations, such as degenerate additive or multiplicative noise.

\subsection*{Acknowledgement}
We thank the anonymous referees for their careful reading or our manuscript and their many insightful comments and suggestions that helped to improve the presentation of the manuscript and simplify several proofs.

\bibliographystyle{plain}
\bibliography{synchronization-refs}

\end{document}